\documentclass[11pt,leqno]{amsart}
\usepackage{amsthm,amsfonts,amssymb,amsmath,oldgerm}
\numberwithin{equation}{section}
\usepackage[utf8]{inputenc}
\usepackage[T1]{fontenc}
\usepackage[english]{babel}
\usepackage{fullpage}
\usepackage{amsmath}
\usepackage{amsfonts}
\usepackage{amssymb}
\usepackage{setspace}
\usepackage{stmaryrd}
\usepackage{mathrsfs}
\usepackage{fancyhdr}
\usepackage{graphicx}
\usepackage{mathabx}
\usepackage{color}
\usepackage{psfrag}
\usepackage{listings} 
\usepackage{float}
\usepackage{url}

\usepackage[top=30mm,bottom=30mm,left=25mm,right=25mm,a4paper]{geometry}

\renewcommand\d{\partial}
\newcommand\dD{\textrm{d}}

\def\eps{\varepsilon }

\newcommand{\I}{{\rm I}}

\newcommand\br{\begin{remark}}
  \newcommand\er{\end{remark}}
\newcommand\bp{\begin{pmatrix}}
  \newcommand\ep{\end{pmatrix}}
\newcommand{\be}{\begin{equation}}
  \newcommand{\ee}{\end{equation}}
\newcommand\ba{\begin{equation}\begin{aligned}}
    \newcommand\ea{\end{aligned}\end{equation}}
\newcommand\ds{\displaystyle}

\newcommand{\beg}{\begin{example}}
  \newcommand{\eeg}{\end{exaplem}}
\newcommand{\bpr}{\begin{proposition}}
  \newcommand{\epr}{\end{proposition}}
\newcommand{\bt}{\begin{theorem}}
  \newcommand{\et}{\end{theorem}}
\newcommand{\bc}{\begin{corollary}}
  \newcommand{\ec}{\end{corollary}}
\newcommand{\bl}{\begin{lemma}}
  \newcommand{\el}{\end{lemma}}
\newcommand{\bd}{\begin{definition}}
  \newcommand{\ed}{\end{definition}}
\newcommand{\brs}{\begin{remarks}}
  \newcommand{\ers}{\end{remarks}}

\newtheorem{theorem}{Theorem}[section]
\newtheorem{proposition}[theorem]{Proposition}
\newtheorem{corollary}[theorem]{Corollary}
\newtheorem{lemma}[theorem]{Lemma}

\theoremstyle{remark}
\newtheorem{remark}[theorem]{Remark}
\theoremstyle{definition}
\newtheorem{definition}[theorem]{Definition}

\newtheorem{example}[theorem]{Example}

\newcommand\R{\mathbf R}

\newcommand\T{\mathbf T}

\newcommand{\Z}{\mathbf Z}

\newcommand{\lla}{\left\langle}
\newcommand{\rra}{\right\rangle}

\newcommand{\mypar}{{\mkern3mu\vphantom{\perp}\vrule depth 0pt\mkern2mu\vrule depth 0pt\mkern3mu}}

\newcommand{\norm}[3]{\|#3\|_{#1,#2}}
\newcommand{\dnorm}[3]{D_{#1,#2}(#3)}
\newcommand{\gampar}[1]{\gamma_{\mypar,#1}}
\newcommand{\gamperp}[1]{\gamma_{\perp,#1}}
\newcommand{\betpar}[1]{\beta_{\mypar,#1}}
\newcommand{\betperp}[1]{\beta_{\perp,#1}}
\newcommand{\Apar}{\d_{v_\mypar}}
\newcommand{\Cpar}{\d_{x_\mypar}}
\newcommand{\Aperp}{\nabla_{v_\perp}}
\newcommand{\Cperp}{\nabla_{x_\perp}}

\newcommand{\TT}{{\mathbf T}}

\newcommand\un{{\underline n}}

\newcommand\upsi{{\underline \psi}}

\newcommand\cC{{\mathcal C}}

\newcommand\cH{{\mathcal H}}

\newcommand\cM{{\mathcal M}}

\newcommand\cO{{\mathcal O}}

\newcommand\tf{{\widetilde f}}

\newcommand\tphi{{\widetilde \phi}}

\newcommand\tpsi{{\widetilde \psi}}

\newcommand\tN{\widetilde{N}}

\DeclareMathOperator*{\esssup}{ess\,sup}

\title{
  Anisotropic Boltzmann-Gibbs dynamics of strongly magnetized Vlasov-Fokker-Planck equations
}

\author{Maxime HERDA}
\address{\flushleft
  \underline{Maxime Herda}\\[.3em]
  Inria, Univ. Lille, CNRS, UMR 8524 – Laboratoire Paul Painlevé\\[.2em]
  F-59000 Lille, France\\[.2em]
  Email: \texttt{maxime.herda@inria.fr}}

  \author{L.Miguel Rodrigues}
  \address{\flushright
    \underline{Luis Miguel Rodrigues}\\[.3em]
    Universit\'e de Rennes 1,\\[.2em]
    IRMAR, UMR CNRS 6625,\\[.2em]
    263 avenue du General Leclerc;\\[.2em]\flushright
    F-35042 Rennes Cedex, FRANCE\\[.2em]
    E-mail: \texttt{luis-miguel.rodrigues@univ-rennes1.fr\\}}
    \thanks{Research of L.Miguel Rodrigues was partially supported by the ANR project
      BoND ANR-13-BS01-0009-01.\\
    }
    \date{\today}

\begin{document}
      \begin{abstract}
	We consider various sets of Vlasov-Fokker-Planck equations modeling the dynamics of charged particles in a plasma under the effect of a strong magnetic field. For each of them in a regime where the strength of the magnetic field is effectively stronger than that of collisions we first formally derive asymptotically reduced models. In this regime, strong anisotropic phenomena occur ; while equilibrium along magnetic field lines is asymptotically reached our asymptotic models capture a non trivial dynamics in the perpendicular directions. We do check that in any case the obtained asymptotic model defines a well-posed dynamical system and when self consistent electric fields are neglected we provide a rigorous mathematical justification of the formally derived systems. In this last step we provide a complete control on solutions by developing anisotropic hypocoercive estimates.
	
	\medskip
	
	
	\smallskip
	\noindent \textsc{MSC2010 subject classifications:} 35Q83, 35Q84, 82D10, 76X05.
      \end{abstract}
      
      \maketitle
      
      \tableofcontents
      
      \section{Introduction}
      
      In the present contribution, we are interested in understanding interactions between oscillations and dissipative mechanisms when the latter are not the dominant effects. We are essentially motivated by the modeling of a strongly magnetized, weakly collisional three-dimensional plasma.
      
      The dynamics of a plasma, a charged gas composed of ions and electrons evolving under self-consistent and external electromagnetic interactions, can be accurately modeled by a system of kinetic equations of Vlasov-type. However, in applications, these models are rarely numerically simulated in their full  complexity because of computer-time costs. Instead, some reduced models are used in computational experiments. Even at a theoretical level these models are precious to capture coherent dynamics emerging on large-time scales from short-scale oscillations or out of dissipation. Therefore it is crucial to understand how they derive from modeling assumptions in some asymptotic regimes. In tokamak physics \cite{miyamoto_2006_plasma}, a usual reduction consists in approximating the density of electrons by Maxwell distribution in velocity and the so-called Boltzmann-Gibbs macroscopic density which reads
      \be
      n_\text{BG}(t,x)\ =\ \frac{1}{Z(t)}\,\exp\left(\frac{q\,\phi(t,x)}{k_BT(t)}\right),
      \label{BGgeneral}
      \ee
      where $t$ and $x$ are respectively the time and space variables, $q$ and $k_B$ stand for the elementary charge and the Boltzmann constant, $\phi$ is the electric potential and $T$ is a temperature. It depends also on a normalizing density $Z$ that could actually be taken to be $1$ by a suitable normalization of the electric potential. The approximation by Maxwellian distribution hinges directly on entropy dissipation arguments whereas the derivation of the Boltzmann-Gibbs approximation stems from the interplay between transport and collisions in an appropriate scaling. Depending on the modeling of both phenomena one obtains different asymptotic descriptions reflecting more or less complex dynamics for $T$. For instance, recently, in \cite{bardos_2016_maxwell}, Bardos, Golse, Nguyen and Sentis formally derive~\eqref{BGgeneral} jointly with an evolution rule for $T$ when starting from a magnetic-field-free kinetic description of electrons and prove that the corresponding evolution obtained by coupling with a full kinetic model for ions dynamics possesses global weak solutions. In any case the asymptotic reduction hinges on essentially the same mechanisms driving large-time\footnote{Here in the sense of $t\to\infty$.} asymptotics for similar models. On the latter the reader is referred to \cite{bouchut_1995_on,HR_hypo}.
      
      In this paper, we are rather interested in a detailed analysis of the normalizing factor $C:=1/Z$ for strongly magnetized plasma. While, as already pointed out, in many applications in tokamak physics the isotropic \eqref{BGgeneral} is often used, namely $C$ at $(t,x)$ is $C(t)$, we claim that under the influence of a strong three-dimensional magnetic field the relevant asymptotic form for the slaved density \eqref{BGgeneral} features a space dependent density $C(t,x_\perp)$ where $x_\perp$ is the projection of $x$ in the plane orthogonal to magnetic field lines. In other words we claim that in a strongly magnetized regime Boltzmann equidistribution of energy occur in velocity variables and spatially along the magnetic field but not perpendicular to it. Therefore we shall call the resulting approximation \emph{anisotropic Boltzmann-Gibbs approximation}. Note that when $C$ is space-dependent it cannot be eliminated by a suitable convention change in the definition of the electric potential. 
      Some similar questions have already been investigated by formal asymptotics 
      both in plasma physics and applied mathematics communities. We point out the articles of Negulescu, Possaner and their collaborators \cite{deCecco-Negulescu-Possanner, negulescu_2016_closure} to the attention of the reader as both an instance of a formal analysis similar to that developed here  and a source of references to the relevant plasma physics literature. Besides the formal derivation of a suitable anisotropic version of \eqref{BGgeneral} and of the corresponding spatio-temporal evolution for $C$, \emph{our goal is to provide various rigorous validations of our formal analysis and a mathematical analysis of the reduced dynamics for $C$}.
      
      \subsection{Our original kinetic model and its approximation}
      To keep the general picture as simple as possible we restrict our attention to a uniform magnetic field, compact spatial domains without boundaries and collisions with a thermal bath, which will yield a constant temperature $T$ in our anisotropic version of \eqref{BGgeneral}.   Our initial object of study is the evolution of a distribution function $f^\eps:\,\R_+\times\T^3\times\R^3 \rightarrow \R_+$, $(t,x,v)\mapsto f^\eps(t,x,v)$
      depending on time $t$, space $x$ and velocity $v$ (where $\T=\R/\Z$ is a one-dimensional torus) according to a Vlasov--Fokker-Planck equation (VFP) 
      \begin{equation} \label{FPf}
	\begin{aligned}
	  \\
	  \eps\ \d_t f^\eps + v\cdot\nabla_xf^\eps+\sigma\,E^\eps\cdot\nabla_vf^\eps - \frac{1}{\varepsilon}\sigma\,v^\perp\cdot\nabla_vf^\eps\, =\, \frac{1}{\eps^\alpha}
	  \textrm{div}_v(vf^\eps + \nabla_vf^\eps),\\~
	\end{aligned}
      \end{equation}
      where $v^\perp=(0,0,1)\wedge v$ --- in coordinates $(v_1,v_2,v_3)^\perp=(-v_2,v_1,0)$ --- stems from the Laplace force induced by the external magnetic field, $\sigma$ stands for the sign of the particle charge, $E^\eps$ is an electric field and $\alpha\in\R$ measures the collisions strength with respect to magnetic effects. Note that as is implicit above, the vector $(0,0,1)$ provides the direction of the magnetic field. Depending on the situation, we shall consider various forms of prescription for the electric field $E^\eps:\R_+\times\T^3\to\R^3$. In any case it derives from an electric potential $\phi^\eps:\R_+\times\T^3\to\R$, $E^\eps=-\nabla_x\phi^\eps$. The simplest situation corresponds to the case where the electric potential itself is taken as a known applied electric potential. For obvious reasons we shall refer to the latter case as the external field case or the linear case. In this case the electric potential is independent of $\eps$ and we sometimes drop the superscript $\eps$ on it. Two other situations are considered in the present paper. Indeed though we have so far mostly discussed motivations from asymptotic reduction in the electronic dynamics we also investigate a similar question for the ions dynamics. This results in two separated prescriptions of the electric field through modeling considerations that are only briefly summarized hereafter but detailed in Section~\ref{s:physics}.
      
      \medskip
      
      In the first nonlinear case, that we shall refer to as the electronic or the light species case, $f^\eps$ is the electron distribution function and phenomena are observed on ions characteristic scales. Consistently we should set $\sigma=-1$ here but we shall mostly keep $\sigma$ undetermined to offer a treatment as unified as possible of all situations. In this first case the scaling parameter $\eps$ represents the square root of the ratio $m_e/m_i$ of masses of electrons and ions. In the asymptotics $\eps\rightarrow 0$, particles are massless zero-inertia electrons and the corresponding dynamics is sometimes called adiabatic evolution. Note that here the fact that magnetic effects are \emph{strongly} perceived by electrons arise from the strong separation in time scales induced by a small mass ratio. As we have already stressed this limit plays a particular important role in the understanding of plasma dynamics and accordingly has been investigated in many similar but distinct situations \cite{herda_2015_massless, badsi_2016_modelling, negulescu_2016_closure, deCecco-Negulescu-Possanner, degond_disparate-mass, degond-lucquin_1996_collision, degond-lucquin_transport,  Bouchut_VP_small-electron-mass, bardos_2016_maxwell}. In a non relativistic regime, the electric field satisfies the classical Poisson equation
      \be\label{Poisson-electron}
      -\delta^2\Delta \phi^\eps\,=\,\sigma(n^\eps-\un),
      \ee
      where $n^\eps:\R_+\times\T^3\to\R_+$ is the macroscopic density of particles under consideration
      $$
      n^\eps(t,x)\,=\,\int_{\R^3}\,f^\eps(t,x,v)\,\dD\,v
      $$
      and $\un:\R_+\times\T^3\rightarrow\R$ denotes the background time-dependent density of the other species with opposite charges.
      
      \medskip
      
      In the second nonlinear model, called the ionic or the heavy species case in the following, $f^\eps$ denotes the ion distribution function. Scaling parameter $\eps$ stands for a scaled Larmor radius which directly measures the strength of the external magnetic field. In this regime, classical Boltzmann-Gibbs approximation for the macroscopic electron density leads to the so-called Poisson-Boltzmann equation for the electric field
      \be\label{Poisson-ion}
      -\delta^2\Delta \phi^\eps\,=\,\sigma\left(n^\eps-\frac{1}{\int_{\T^3} e^{\sigma\phi^\eps(y)}\,\dD\,y}\,e^{\sigma\phi^\eps}\right)\,.
      \ee
      In non-collisional settings the strong magnetic field limit $\eps\rightarrow0$ has also been thoroughly investigated by many authors \cite{degond_2016_asymptotic, Golse-StRaymond_strong-magnetic, StRaymond_large-velocities, frenod_1998_homogenization,Cheverry_whistler,Cheverry_axisymmetric,Miot_gyrokinetic,Possanner,Bostan_2D-VP,FR18}. As in those pieces of work time scale is chosen here to observe a non trivial averaged dynamics.
      
      \medskip

      Concerning the strength of collisions, we focus on cases where $|\alpha|<1$, $\alpha<0$ corresponding to a regime of evanescent collision frequency and $\alpha>0$ to a strongly collisional regime. The assumption $\alpha>-1$ ensures that even in weakly collisional regimes collisions are strong enough to lead on the large observation time scales we consider to a global equilibrium if magnetic effects were discarded (see \cite{bouchut_1995_on,HR_hypo} for a related analysis and Remark~\ref{BGwhy} for further comments) so that the non-equilibrium features that we shall describe are indeed due to the presence of a strong magnetic field. The condition $\alpha<1$ enforces that magnetic effects are dominant. We refer the reader to \cite{herda_2015_massless} for an extensive list of references concerning the borderline case $\alpha=1$ that yields a diffusive limit when $\eps\to0$ which provides an asymptotic model retaining some features of the external magnetic field. We point out that in contrast with the case under study here, in the borderline case the specific geometry of the magnetic field plays almost no role in the derivation of an asymptotic model.
      
      \bigskip
      
      In Section~\ref{s:heuristic} by arguing on formal grounds we provide a heuristic argument covering all cases introduced hereinabove and in Section~\ref{s:compactness} we offer a rigorous justification in the linear case that both suggest that, for any family $(f^\eps)_{\eps>0}$ of suitable solutions to \eqref{FPf} with a given initial datum $f_0$, taking the limit $\eps\to0$ yields a limiting distribution $f$ and a limiting electric potential $\phi$ that have reached an adiabatic regime along the magnetic field
      \be
      f(t,x,v)\,=\,N(t,x_\perp)\,\frac{e^{-\sigma\phi(t,x)}}{\int_{\T} e^{-\sigma\phi(t,x_\perp,y_\mypar)}\,\dD\,y_\mypar}\,M(v)
      \label{fMB}
      \ee
      where $x=(x_\perp,x_\mypar)\in\T^2\times\T$ and $M$ is a Maxwellian distribution
      $$
      M(v) = \frac{1}{(2\pi)^{3/2}}\,e^{-\frac{|v|^2}{2}}.
      $$
      In this asymptotic regime the dynamics is slaved to the evolution of a reduced macroscopic distribution $N:\R_+\times\T^2\to\R_+$ in the perpendicular direction, satisfying
      \be\label{guiding-center}
      \d_t N\,+\,\textrm{div}_{x_\perp}(N\,(\nabla_{x_\perp}\tphi)^\perp)\,=\,0
      \ee
      with initial condition
      $$
      N(0,x_\perp)\, =\, N_0(x_\perp)\,=\,\iint_{\T\times\R^3} f_0(x_\perp,x_\mypar,v)\,\dD\,x_\mypar\,\dD\,v\,,
      $$
      where $\tphi:\R_+\times\T^2\to\R$ is an $x_\mypar$-averaged version of $\phi$
      \be\label{averaged-phi}
      \tphi(t,x_\perp)\,=\,-\sigma\ln\left(\int_\T e^{-\sigma\phi(t,x_\perp,x_\mypar)}\dD\,x_\mypar\right)\,,
      \ee
      and $\phi$ is either the initially prescribed electric field in the external field case, or is obtained by solving \eqref{Poisson-electron} or \eqref{Poisson-ion} (dropping the superscript $\eps$) with the anisotropic Boltzmann-Gibbs density
      \be\label{slaved-n}
      n(t,x)\,=\,N(t,x_\perp)\,\frac{e^{-\sigma\phi(t,x)}}{\int_{\T} e^{-\sigma\phi(t,x_\perp,y_\mypar)}\,\dD\,y_\mypar}\,=\,N(t,x_\perp)\,e^{-\sigma\left(\phi(t,x)-\tphi(t,x_\perp)\right)}\,.
      \ee
      To ease comparison with the initial discussion we observe that \eqref{slaved-n} provides an anisotropic version of \eqref{BGgeneral} where $C=1/Z$ would be given by
      $$C(t,x_\perp)=N(t,x_\perp)/(\int_{\T} e^{-\sigma\phi(t,x_\perp,y_\mypar)}\,\dD\,y_\mypar).$$
      
      The reader may rightfully wonder whether the reduced dynamics described above could in disguise follow a set of purely \emph{two-dimensional} local differential equations. The answer turns out to be positive in the heavy-species case that actually follows
      $$
      \d_t N\,+\,\textrm{div}_{x_\perp}(N\,(\nabla_{x_\perp}\tphi)^\perp)\,=\,0\,,
      \qquad\qquad
      -\delta^2\Delta_{x_\perp}\tphi\,=\,\sigma\left(N-\frac{e^{\sigma\tphi}}{\int_{\T^2}e^{\sigma\tphi(\cdot,y_\perp)}\dD y_\perp}\right)\,.
      $$
      In the light-species case the presence of a three-dimensional background $\un$ prevents any fully two-dimensional reduction from happening. However in the special case where actually $\un$ does not depend on $x_\mypar$ the light-species case does reduce to a two-dimensional set of equations, the classical guiding center model
      \[
      \d_t N\,+\,\textrm{div}_{x_\perp}(N\,\nabla_{x_\perp}^\perp\phi)\,=\,0\,,\qquad\qquad
      -\delta^2\Delta_{x_\perp}\phi\, =\, \sigma(N-\un)\,,
      \]
      that has been analytically derived by Golse and Saint-Raymond \cite{Golse-StRaymond_strong-magnetic, StRaymond_large-velocities} and Miot \cite{Miot_gyrokinetic} as the large magnetic field limit of the Vlasov-Poisson system with uniform magnetic fields. Interestingly enough it is also consistent with formal arguments in the recent \cite{degond_2016_asymptotic} by Degond and Filbet that in a non-uniform large magnetic field regime obtain from the Vlasov-Poisson system a set of equations for a limiting $f(t,x,v)=F(t,x,|v_\perp|,v_\mypar)$ which when specialized to a constant magnetic field as considered here is
      $$
      \begin{array}{l}\ds
	\d_t F + \nabla_{x_\perp}^\perp\phi\cdot\nabla_{x_\perp}F -\sigma\,\partial_{x_\mypar}\psi\,\d_{v_\mypar}F + v_\mypar\,\d_{x_\mypar}P - \sigma\,\d_{x_\mypar}\phi\,\partial_{v_\mypar}P\,=\,0
	\\[0.25em]\ds
	v_\mypar\,\d_{x_\mypar}F - \sigma\d_{x_\mypar}\phi\,\d_{v_\mypar}F\,=\,0\,,
	\quad
	-\Delta_x\phi\,=\,\sigma\Big((2\pi)\int_{\R_+\times\R}F(\cdot,\cdot,p,v_\mypar)\,p\,\dD p\,\dD v_\mypar-\un\Big)
	\\\ds
	-\Delta_x\psi\,=\,\sigma\Big((2\pi)\int_{\R_+\times\R}P(\cdot,\cdot,p,v_\mypar)\,p\,\dD p\,\dD v_\mypar-\un\Big)\,
      \end{array}
      $$
      where $P$ is thought as a multiplier associated with the constraint on $F$ encoded by the first equation of the second-line. \emph{A priori} for collisionless plasmas nothing seems to force a Maxwellian distribution in velocity but if one inserts a Maxwellian distribution $f(t,x,v)=n(t,x)M(v)$ in the asymptotic model formally derived by Degond and Filbet then the first equation of the second-line of the system becomes \eqref{slaved-n} for some $N$ and, afterwards, integrating in $(x_\mypar,|v_\perp|,v_\mypar)$ the first equation to eliminate $(P,\psi)$ reduces the two first lines of the system to \eqref{guiding-center}-\eqref{averaged-phi}-\eqref{slaved-n}-\eqref{Poisson-electron}.
      
      \subsection{Main analytical results}
      
      We now state our two sets of mathematical results, on one hand existence results for nonlinear asymptotic models and on the other hand analytic proofs of validity for the linear asymptotic model.
      
      Roughly speaking, nonlinear asymptotic systems remotely look like transport equations for $N$ by a divergence-free vector-field obtained by applying to $N$ a linear differential operator of order $-1$. Yet in both cases the relation between $N$ and $\tphi$ is far from being linear and in the light-species case the system is not even truly two-dimensional so that it is not readily apparent that those nonlinear asymptotic systems do enjoy existence and well-posedness results expected for the foregoing class of equations. In Section~\ref{s:WP} we prove that it is the case at least in subcritical\footnote{However we let open more technical issues such as existence of renormalized solutions and well-posedness in critical spaces.} regimes and our strategy does follow the above analogy. Indeed once sufficient control on the construction of $\tphi$ has been obtained one may follow arguments used for similar equations such as Vlasov-Poisson equations \cite{Glassey,bouchut_2000_kinetic,Rein} or two-dimensional incompressible Euler equations in vorticity formulation \cite{Marchioro-Pulvirenti,Chemin,Majda-Bertozzi}. Consequently, the key part of the argument is a detailed analysis of existence, uniqueness, continuous dependence and regularity results for solutions to both elliptic equations \eqref{Poisson-ion}-\eqref{slaved-n} and \eqref{Poisson-electron}-\eqref{slaved-n} (as well as estimates on $\tphi$ in terms of $\phi$) that are for $\tphi$ as two-dimensional and linear as possible. From these we obtain the following results for the full asymptotic systems.
      
      \bt[Strong solutions]\label{th:strong}$ $\\
      Let $p>2$, $\delta>0$ and $\un\in \cC_b(\R_+;W^{1,p}(\T^3))$ with mean constant equal to one, $\un\geq0$.\\
      For any $N^\text{in}\in W^{1,p}(\T^2)$, $N^\text{in}\geq0$, $\int_{\T^2}N^\text{in}=1$, there exists a unique maximal solution $(N,I)$ to \eqref{guiding-center}-\eqref{averaged-phi}-\eqref{Poisson-electron}-\eqref{slaved-n} (\textit{resp.} \eqref{guiding-center}-\eqref{averaged-phi}-\eqref{Poisson-ion}-\eqref{slaved-n}) in $\cC(I;L^1(\T^2))\cap \cC_w(I;W^{1,p}(\T^2))$ starting from $N^\text{in}$ at time $0$. Moreover the maximal time of existence is uniformly bounded away from zero on bounded sets of $W^{1,p}(\T^2)$.
      \et
      
      In the former Theorem, $\cC_w$ denotes the space of functions that are continuous for the weak topology on the space of values and of course in the heavy-species case no $\un$ is actually needed.
      
      \bt[Weak solutions]\label{th:weak}
      Let $p>4/3$, $\delta>0$.\\
      Let $\un\in \cC(\R;L^{6/5}(\T^3))\cap L^\infty_{loc}(\R;L^p(\T^3)\cap L^{3/2,1}(\T^3))$ with mean constant equal to one, $\un\geq0$.\\
      For any $N^\text{in}\in L^p(\T^2)$, $N^\text{in}\geq0$, $\int_{\T^2}N^\text{in}=1$, there exists a global weak solution $N$ to \eqref{guiding-center}-\eqref{averaged-phi}-\eqref{Poisson-electron}-\eqref{slaved-n} (\textit{resp.} \eqref{guiding-center}-\eqref{averaged-phi}-\eqref{Poisson-ion}-\eqref{slaved-n}) starting from $N^\text{in}$ at time $0$. 
      \et
      
      The reader unfamiliar with Lorentz spaces as appearing in our last statement may replace $L^p(\T^3)\cap L^{3/2,1}(\T^3)$ with $L^q(\T^3)$ for some $q>\max(\{p,3/2\})$ to obtain a simpler weaker statement or consult \cite[Chapter~2]{Lemarie-Rieusset}.
      
      Now we elaborate on the analogy mentioned in the paragraph preceding our statements to motivate critical thresholds appearing in our foregoing theorems. For strong solutions of transport equations a natural threshold is obtained by requiring the Lipschitz property for the advection field as it enables to transport regularity. Assuming that $(\nabla_{x_\perp}\tphi)^\perp$ is indeed one-derivative more regular than $N$ this amounts here to $W^{2,p}(\T^2)\hookrightarrow W^{1,\infty}(\T^2)$, hence $p>2$. Concerning weak solutions, natural thresholds are obtained by requiring that all terms in the equations should be well-defined as distributions and this is usually an issue only for nonlinear terms, hence here only for $N(\nabla_{x_\perp}\tphi)^\perp$. Therefore we seek $\nabla_{x_\perp}\tphi \in L^{p'}(\T^2)$ since $N\in L^p(\T^2)$ (where $(\cdot)'$ denotes Lebesgue conjugation, hence $1/p+1/p'=1$).  Assuming again that $(\nabla_{x_\perp}\tphi)^\perp$ is indeed one-derivative more regular than $N$ this leads to enforce $W^{1,p}(\T^2)\hookrightarrow L^{p'}(\T^2)$ which yields $p=4/3$ as a critical exponent. Note that in our actual analysis we need to go from $N$ to $\tphi$ through $\phi$ and thus need to replace two-dimensional embeddings implicitly used above with suitable anisotropic versions of three-dimensional embeddings to preserve two-dimensional-like scalings.
      
      Though we do not explicitly state it here an inspection of our proofs also provides propagation of regularity for strong solutions and weak-strong uniqueness. We do not elaborate here on the crucial role of $\delta$, that we consider as fixed in the present contribution, on lower bounds for existence time of strong solutions but we refer the reader to \cite{HR_hypo} for a detailed discussion and a thorough analysis of a non magnetized version of our original systems.
      
      \medskip
      
      Second, in the external field case, we prove in Section~\ref{s:maths} that formal arguments may indeed be replaced with sound rigorous mathematical analysis. We provide such a justification for what is often considered to be the weakest notion of solutions that are physically relevant, that is, for free energy solutions of \eqref{FPf} in the sense of  Bouchut and Dolbeault \cite{bouchut_1995_on, dolbeault_1999_free}. More explicitly we consider all $L\log L$ distributions with finite second moments in velocity or equivalently, distributions with bounded mass, finite kinetic energy and finite entropy.
      
      \begin{theorem}\label{th:compactness}
	Assume that $\phi\in W^{1,1}_{loc}(\R_+;L^\infty(\T^3))$ and $f_0$ is such that
	\[
	f_0\geq 0,\quad \iint_{\T^3\times\R^3}\left(1 + |v|^2 + \ln_+ f_0\right)\,f_0\, \dD x \dD v<\infty
	\]
	(where $(\cdot)_+$ denotes positive part). Then there exists a unique $f^\eps\in \cC(\R_+;L^1(\T^3\times\R^3))$ solving Equation~\eqref{FPf} and starting from $f_0$ at time $0$. Moreover
	there exists $N\in L^\infty(\R_+;L^1(\T^2))$ and a positive sequence $(\eps_n)$ converging to zero such that when $n\rightarrow \infty$
	\[
	f^{\eps_n}(t,x,v)\rightarrow N(t,x_\perp)\,\frac{e^{-\sigma\phi(t,x)}}{\int_{\T} e^{-\sigma\phi(t,x_\perp,y_\mypar)}\dD y_\mypar}\,M(v)\qquad\text{weakly in }L^1_\text{loc}(\R_+;L^1(\T^3\times\R^3))\,,
	\]
	and $N$ solves the transport equation \eqref{guiding-center}-\eqref{averaged-phi}.
	\label{mainCV}
      \end{theorem}
      
      Note that when $\phi$ is sufficiently regular to ensure uniqueness for the transport equation one may actually replace convergence along the subsequence $(\eps_n)$ with full convergence as $\eps\to0$ by the classical argument relying on compactness and uniqueness of limit points.
      
      Finally, we want to obtain convergence rates in strong norms for asymptotic limits that are expected to be compatible with strong convergence. Indeed one expects three kind of phenomena to co-exist, dissipation by collisions leads to Maxwellian distribution in velocity, interaction between transport and collisions yields anisotropic Boltzmann-Gibbs behavior of macroscopic density for some $N^\eps$, strong oscillations induced by the magnetic field average $N^\eps$ to some $N$ solving \eqref{guiding-center}. By essence the last part does not lend itself to a strong convergence analysis. For the other ones we develop a form of anisotropic hypocoercive estimates that provide a strong form of control on norms involved in the limiting process. Note that even for the simplest part of the convergence analysis leading to Maxwellian behavior one needs to assume some strong form of localization of initial data to derive decay rates. Accordingly we assume our initial data to lie in some weighted functional space. The justification of the Boltzmann-Gibbs approximation also requires control on derivatives of the solutions but we assume no regularity on the initial data and instead gain this regularity at later time from hypoelliptic properties of the dynamics. At a technical level our proof proceeds by building a suitable dissipated functional that captures both hypocoercivity observed on relevant commutators as in \cite{Villani_hypo,HR_hypo} and hypoellipticity as in \cite{Herau_FP-confining, Herau-Nier,HR_hypo}. For simplicity we assume $\phi$ to be time-independent here.
      
      \begin{theorem}[Convergence rates]\label{th:hypo}
	Let $\phi\in W^{2,\infty}(\T^3)$.\\
	Then there exists a positive constant $C$ such that for any $\eps\in(0,1)$ and any initial data $f_0\in L^2(M^{-1}(v)\,\dD x\,\dD v)$, the strong solution $f^\eps$ solving \eqref{FPf} and starting from $f_0$ satisfies
	\[
	\begin{array}{rcl}\ds
	  \left\|(t,x,v)\mapsto
	  f^\eps(t,x,v)-n^\eps(t,x)M(v)\right\|_{L^2(\R_+;\ L^2(M^{-1}(v)\dD x\dD v))}&\leq&\ds
	  C\,\|f_0\|_{L^2(M^{-1}(v)\dD x\dD v)}\,\eps^{\frac{\alpha+1}{2}}\\ \ds
	  \left\|(t,x)\mapsto n^\eps(t,x)-N^\eps(t,x_\perp)\,\frac{e^{-\sigma\phi(x)}}{\int_{\T} e^{-\sigma\phi(x_\perp,y_\mypar)}\,\dD\,y_\mypar}\right\|_{L^2(\R_+\times\T^3)}
	  &\leq&\ds C\,\|f_0\|_{L^2(M^{-1}(v)\dD x\dD v)}\,\eps^{\frac{1-|\alpha|}{2}}
	\end{array}
	\]
	where
	$$
	n^\eps\,=\,\int_{\R^3} f^\eps(\cdot,\cdot,v)\dD v
	\qquad\textrm{and}\qquad
	N^\eps\,=\,\int_{\T}n^\eps(\cdot,\cdot,y_\mypar)\dD y_\mypar\,.
	$$
	In particular, as $\eps\rightarrow 0$ the distribution function $f^\eps$ gets $\eps^{\frac{1-|\alpha|}{2}}$-close to the anisotropic Maxwell-Boltzmann density 
	$$
	(t,x,v)\mapsto N^\eps(t,x_\perp)\,\frac{e^{-\sigma\phi(x)}}{\int_{\T} e^{-\sigma\phi(x_\perp,y_\mypar)}\,\dD\,y_\mypar}\,\frac{e^{-\frac{|v|^2}{2}}}{(2\pi)^{3/2}} 
	$$
	in $L^2(\R_+;\ L^2(M^{-1}(v)\dD x\dD v))$.
      \end{theorem}
      
      A few comments on rates are in order. Rates in our statement are expected to be optimal. Indeed the first rate is naturally associated with the fact that solutions to 
      $$
      \eps\ \d_t \tf^\eps \, =\, \frac{1}{\eps^\alpha}\textrm{div}_v(v\tf^\eps + \nabla_v\tf^\eps)
      $$
      starting from $\tf_0\in L^2(M^{-1}(v)\dD v)$ decay to a Maxwellian distribution at a rate $e^{-t/\eps^{1+\alpha}}$ in $L^2(M^{-1}(v)\dD v)$. 
      On the other hand the second one is closely connected to the convergence of solutions to
      $$
      \eps\ \d_t f^\eps + v\cdot\nabla_xf^\eps+\sigma\,E^\eps\cdot\nabla_vf^\eps\, =\, \frac{1}{\eps^\alpha}\textrm{div}_v(vf^\eps + \nabla_vf^\eps)
      $$
      starting from $f_0\in L^2(M^{-1}(v)\,\dD x\,\dD v)$ such that $\nabla_xf_0,\nabla_vf_0, vf_0\in L^2(M^{-1}(v)\,\dD x\,\dD v)$ towards a Maxwell-Boltzmann distribution at a rate $e^{-{\kappa_0\,t}/{\eps^{1-|\alpha|}}}$ in $L^2(M^{-1}(v)\dD x\dD v)$ for some $\kappa_0>0$. Note that the latter rate is not monotone with respect to $\alpha$. This reflects that collisions are not the only mechanism involved in this convergence. For a thorough discussion we refer the reader to \cite{HR_hypo} that analyzes nonlinear magnetic-field-free analogous questions.

      \subsection*{Structure of the paper}
      
      The structure of the paper follows the organization of our introduction. In Section~\ref{s:formal} we provide motivations for our initial models and a formal derivation of reduced equations. In Section~\ref{s:WP} we analyze nonlinear reduced asymptotic models to prove Theorems~\ref{th:strong} and~\ref{th:weak}. Then in Section~\ref{s:maths} we prove Theorems~\ref{th:compactness} and~\ref{th:hypo} hence justify the linear reduced model. At last in Section~\ref{sec:conclu}, we provide some further technical comments and perspectives on our results.
      
      \section{Heuristic considerations}\label{s:formal}
      
      In the following we give modeling and heuristic justifications for both our starting and limiting models. We explain in Section \ref{s:physics} how our initial physical models --- Equation~\eqref{FPf} completed with some prescription of the electric field --- arise precisely in the regime $\eps\ll 1$ from more complete two-species systems. Then, in Section \ref{s:heuristic}, starting from \eqref{FPf} we argue on formal grounds to identify the asymptotic dynamics obeying \eqref{guiding-center}-\eqref{averaged-phi}-\eqref{slaved-n}.
      
      \subsection{Original physical variables}\label{s:physics}
      
      Though our goal is mostly to identify qualitative mechanisms our choice of systems originates in concrete realistic plasma dynamics, from which they are obtained by a combination of arguments --- that we expound now --- either of asymptotic analysis type or made purely to lower technicalities to their bare minimum.
      
      We begin our discussion by considering the full dynamics of a plasma containing electrons with negative charge $-q$ and mass $m_e$ and ions of mass $m_i$ and charge $Zq$, where $q$ is the elementary charge and $Z$ is the atomic number of ions. In the following we take into account an external unidirectional magnetic field of constant amplitude $\bar{B}$, to be thought of as a confinement field for a fusion device, but as is classical in this regime (and rigorously justified in some closely related contexts, see for instance \cite{degond_1986_local}) we neglect self-induced magnetic effects. To reduce technicalities in the analysis we also idealize the mechanism of collisions between particles and describe collisions as if they were occurring with a thermal bath with temperature $\theta$ and zero mean velocity. To our opinion this is by far the less realistic of our simplifications and the only one that does not bear principally on modeling considerations but we expect that the phenomena that we identify in our simplified evolution do occur on a much wider range of models. In this direction we stress that the formal asymptotic analysis carried out in Section \ref{s:heuristic} does not rely heavily on the specific form of the full collision operator but only on the nature of its kernel, here Maxwellian equilibria with zero mean velocity and fixed temperature.
      
      Denoting respectively by $f_e$ and $f_i$ electron and ion distribution functions, the original equations, written in physical variables, are 
      \be
      \left\{\begin{aligned}
	&\partial_tf_e + v\cdot\nabla_xf_e  - \tfrac{q}{m_e}\left(- \nabla_x\phi -\bar{B}\,\,v^\perp \right) \cdot \nabla_vf_e &=&&& \nu_\text{col}^{(e)}\,\nabla_v\cdot \left( vf_e + \tfrac{k_B \theta}{m_e}\nabla_vf_e\right),\\
	&\partial_tf_i + v\cdot\nabla_xf_i  + \tfrac{Zq}{m_i}\left(- \nabla_x\phi -\bar{B}\,\,v^\perp \right) \cdot \nabla_vf_i &=&&& \nu_\text{col}^{(i)}\,\nabla_v\cdot \left( vf_i + \tfrac{k_B \theta}{m_i}\nabla_vf_i\right),\\
	&-\eps_0\,\Delta_x\phi\ =\ Zq\,n_i-q\,n_e,
      \end{aligned}\right.
      \label{FPfull}
      \ee
      where the macroscopic density of the species $s\in\{i,e\}$ is given by
      $
      n_s(t,x) = \int_{\R^3}f_s(t,x,v)\,\dD v.
      $
      The parameter $\nu_\text{col}^{(s)}$ is the characteristic frequency of collisions with the thermal bath for the species $s\in\{i,e\}$, $k_B$ is the Boltzmann constant and $\eps_0$ is the dielectric constant. To ease comparisons of respective sizes of parameters and allow corresponding asymptotic analysis we now aim at turning \eqref{FPfull} in dimensionless form. To do so, in the following, we denote by $L$ the characteristic length of the system, $t_\text{obs}$ the characteristic observation time and $V_i$, $V_e$ the thermal velocity of  ions and electrons respectively. For any other physical quantity $G$, we denote by $\bar{G}$ the characteristic value of $G$ and $G'$ the dimensionless quantity associated to $G$ so that $G = \bar{G}G'$. Accordingly, we introduce
      \[
      f_s(t,x,v)\ =\ \frac{\bar{n}_s}{V_s^3}\,f_s'\left(\frac{t}{t_\text{obs}}, \frac{x}{L}, \frac{v}{V_s}\right),\qquad n_s(t,x)\ =\ \bar{n}_s n'\left(\frac{t}{t_\text{obs}}, \frac{x}{L}\right),\]
      \[\phi(t,x)\ =\ \bar{\phi}\phi'\left(\frac{t}{t_\text{obs}}, \frac{x}{L}\right).
      \]
      where $s\in\{i,e\}$, and reformulate \eqref{FPfull} in terms of $f_i'$, $f_e'$, $n_i'$, $n_e'$ and $\phi'$.
      \subsubsection*{Physical scales} We assume that the plasma is globally neutral and introduce the characteristic number of electrons $N$ defined by 
      \[
      Z\bar{n}_i = \bar{n}_e = N.
      \]
      Moreover we consider a hot plasma \cite{bellan_2006_fundamentals} meaning that characteristic temperatures (or kinetic energy) are equal to that of the thermal bath.
      It implies
      \[
      m_iV_i^2 = m_eV_e^2 = k_B\theta.
      \]
      Besides, the characteristic potential energy due to electric effects is assumed of same order than that of kinetic effects which yields
      \[
      \bar{\phi} = \frac{k_B\theta}{q}.
      \]
      
      Now we may introduce several physical time and space scales characterizing each of the electric, magnetic and collision phenomena.
      Electrostatic constants are the  Debye length and plasma time respectively given by 
      \[
      \lambda_D = \sqrt{\frac{\eps_0k_B \theta}{q^2N}},\qquad\qquad 
      t^{(s)}_{p} = \frac{\lambda_D}{V_s}.
      \]
      They measure the typical length of influence of an isolated particle and the period of electrostatic waves in the plasma.
      Magnetic constants are cyclotron time and the Larmor radius defining the period and radius  of gyration of particles around magnetic field lines and reading respectively
      \[
      t^{(s)}_{c} = \frac{m_s}{q\bar{B}_\text{ext}},\qquad\qquad
      r_L^{(s)} = V_s t_c^{(s)}.
      \]
      Finally the mechanism of collisions is characterized by a typical time between two collisions $1/(\nu_\text{col}^{(s)})$ and the mean free path of a particle
      \[
      l_s = \frac{V_s}{\nu_\text{col}^s}.
      \]
      We refer to the physics literature \cite{goldston_1995_introduction, bellan_2006_fundamentals,miyamoto_2006_plasma} for deeper insights and thorough comments on the respective roles of each of the former quantities. 
      \subsubsection*{Dimensionless parameters} From now on we use ionic quantities as references and set
      \[
      t_\text{obs} = \frac{1}{\nu_\text{col}^{(i)}}\tau_\text{obs}\,,\qquad\qquad L = l_i.                                                                                                       
      \]
      This leads to the consideration of quotients
      \[
      \delta = \frac{\lambda_D}{L},\qquad \mu = \frac{r_L^{(i)}}{L},\qquad \gamma = \frac{l_e}{l_i},\qquad \lambda = \frac{m_e}{m_i}.
      \]
      Since we focus on globally-neutral hot plasmas one may expect the ratio between mean free paths $\gamma$ to depend only on the mass ratio $\lambda$. Consistently we set for some real exponent $\alpha$,
      \[
      \gamma = \lambda^{\alpha/2}.
      \]
      Incidently we point to the attention of the reader \cite[Section 1.5]{degond_disparate-mass} as an instance of a detailed modeling analysis of the dependence of collision frequencies with respect to mass ratio $\lambda$ for more reallistic collisional operators and we note that the modeling arguments there suggests that $\alpha=0$ would be the case of highest practical interest. 
      Going on by dropping primes on dimensionless quantities, we may express the dimensionless form of the equations as
      \be
      \left\{\begin{aligned}
	\frac{1}{\tau_\text{obs}} &\partial_tf_e + \frac{1}{\sqrt{\lambda}}v\cdot\nabla_xf_e  + \frac{1}{\sqrt{\lambda}}\left(\nabla_x\phi +\frac{1}{\sqrt{\lambda}\,\mu}\,v^\perp \right) \cdot \nabla_vf_e \ =\ \frac{1}{\lambda^{(1+\alpha)/2}}\,\nabla_v\cdot \left( vf_e +\nabla_vf_e\right),\\
	\frac{1}{\tau_\text{obs}}  &\partial_tf_i + v\cdot\nabla_xf_i  + Z\left(- \nabla_x\phi -\frac{1}{\mu}v^\perp \right) \cdot \nabla_vf_i\ =\ \,\nabla_v\cdot \left( vf_i +\nabla_vf_i\right),\\
	&-\delta^2\,\Delta_x\phi\ =\ n_i-n_e,
      \end{aligned}\right.
      \label{FPscaled}
      \ee
      We refer to \cite{degond_disparate-mass,badsi_2016_modelling} for more details on similar computations. Since the changes required to deal with the general case are of notational nature from now on we assume that the atomic number of ions $Z$ is equal to one. Note that the foregoing assumption, made here purely for expository purposes, is consistent with the presence of light ions in tokamak plasmas \cite{miyamoto_2006_plasma}. 
      
      In the situations we have in mind both $\lambda$ and $\mu$ are small. In particular, the mass ratio $\lambda$ is of order $10^{-4}$ for a deuterium plasma. As a consequence typical time scales of ions and electrons completely uncouple since, with our previous set of notation,
      \[
      \lambda
      \,=\,\left(\frac{ t_p^{(e)}}{ t_p^{(i)}}\right)^2\,=\,\frac{ t_c^{(e)}}{ t_c^{(i)}}\,.
      \]
      The systems we consider in the core of the present paper are motivated by the asymptotic analysis of the case where $\lambda\ll\mu\ll1$. We believe that this distinguished regime is of particular relevance for applications. We stress however that it is far from covering the full range of $(\lambda,\mu)\ll 1$ and we do not intend here to justify a double limit in any case.
      
      \subsubsection*{Ion dynamics interacting with massless electrons in the strong magnetic field limit.}
      Neglecting the fact that a strong magnetic field could slow down thermalization effects,  one expects, as it is usually assumed in tokamak physics applications, that the electronic zero-inertia regime $\lambda\rightarrow 0$ brings the macroscopic density of electrons to a Boltzmann-Gibbs density provided that the scaled observation time is at least of order one, $\tau_\text{obs}\gtrsim1$. See for instance Remark~\ref{BGwhy} for some elements of justification. This yields 
      \[
      n_e(t,x) = \frac{e^{\phi(t,x)}}{\int_{\T^3}e^{\phi(t,x)}\dD x}. 
      \]
      This approximation being taken for granted, we are now interested in the  limit $\mu\rightarrow 0$, meaning that the Larmor radius becomes infinitely small. We point out that the corresponding asymptotics has already been extensively studied in related contexts, notably for collisionless models in \cite{frenod_1998_homogenization, Golse-StRaymond_strong-magnetic, StRaymond_large-velocities,Miot_gyrokinetic,Possanner,Bostan_2D-VP,FR18}. 
      
      In the present case, the small parameter of interest is thus \[\eps = \mu\]
      and we recover \eqref{FPf}-\eqref{Poisson-ion}, with $\sigma =1$, $\alpha = 0$ on the observation time scale $\tau_\text{obs}=1/\mu$. As already mentioned the later choice reflects the fact that one needs to consider long-time dynamics to observe a significant averaged evolution.
      
      \subsubsection*{Electron dynamics in the massless limit}
      
      Though the former scenario may seem plausible our intention is precisely to convince the reader that the above-mentioned Boltzmann-Gibbs approximation for $n_e$ fails when collisions are not the dominant mechanism leading the electron dynamics. To do so, we focus on the kinetic description of electrons in the massless limit $\lambda\to0$ when $\alpha<1$.
      
      In this context the small parameter of interest is \[\eps = \sqrt{\lambda}.\] 
      Since the ion evolution equation does not depend explicitly on $\lambda$ and the coupling of equations is relatively weak we make the further simplification assumption that the macroscopic density of ions $n_i$ is independent of $\lambda$ and hence can be considered as frozen and given by a background density $\un$. This reflects the fact that without this simplification one would expect to recover in the limit $\lambda\rightarrow0$ the same asymptotic equation derived under the frozen ion density assumption but coupled with the Vlasov-Fokker-Planck equation for the evolution of the "frozen" density $\un$. In slightly different related situations several papers dealing with the full coupling 
      \cite{herda_2015_massless,bardos_2016_maxwell} have proved that it was indeed a sound expectation.

      After the harmless notational simplification consisting in setting $\mu = 1$, we recover the model \eqref{FPf}-\eqref{Poisson-electron}, with $\sigma  = -1$ when $\tau_\text{obs}=1$.
      
      \br
      Our analysis not only indicates that the classical isotropic Boltzmann-Gibbs approximation fails to capture the electron dynamics in the massless limit but also provides an anisotropic replacement that may be coupled with the slow dynamics of ions. For completeness' sake we draw now briefly some consequences for the ion dynamics. Our first claim is that the dynamics of ions surrounded by massless electrons could be correctly described by
      $$
      \left\{\begin{array}{rcl}\ds
	\frac{1}{\tau_\text{obs}}\d_t N_e\,+\,\mu\,\textrm{div}_{x_\perp}(N_e\,(\nabla_{x_\perp}\tphi_e)^\perp)&=&0\\
	\frac{1}{\tau_\text{obs}}\d_tf_i + v\cdot\nabla_xf_i  + \left(- \nabla_x\phi -\frac{1}{\mu}v^\perp \right) \cdot \nabla_vf_i&=&\,\nabla_v\cdot \left( vf_i +\nabla_vf_i\right),\\
	-\delta^2\,\Delta_x\phi&=&n_i-Ne^{\phi-\tphi_e},
      \end{array}\right.
      $$
      where $\tphi_e$ is obtained from $\phi$ by the nonlinear averaging \eqref{averaged-phi} with $\sigma = -1$. An analysis similar to the one expounded here then shows that on time-scale $\tau_{\text{obs}}=1/\mu$ the limit $\mu\to\infty$ leads $n_i$ to an anisotropic Boltzmann-Gibbs regime where corresponding $(N_e,N_i)$ satisfies a system of coupled guiding-center equations
      $$\left\{
      \begin{array}{rcl}\ds
	\d_t N_s\,+\,\textrm{div}_{x_\perp}(N_s\,\nabla_{x_\perp}^\perp\phi)&=&0,\qquad 
	s=e,i,\\\ds
	-\delta^2\,\Delta_{x_\perp} \phi&=&\ds N_i - N_e\,.
      \end{array}\right.
      $$
      Incidentally note that consistently with classical gyrokinetic theory both species are advected by the same electric drifts irrespective of the sign of their charges.
      \er
      \subsection{Formal derivation}\label{s:heuristic}
      
      Now we begin the derivation of \eqref{guiding-center}-\eqref{averaged-phi}-\eqref{slaved-n} from \eqref{FPf}. Here we argue formally and in particular take for granted in this heuristic section that any sequence of functions that is bounded in one norm actually converges in any norm required to take limits and that rates of convergence may indeed be predicted by balancing orders in equations. It goes without telling that conclusions thus derived must then be taken with care. In particular even when formally derived results may be rigorously justified, in general the mathematical validation may not proceed by proving each step of the crude path followed arguing heuristically. Moreover especially in strongly oscillating asymptotics formal arguments may lead to misleading conclusions. Indeed terms that may seem prominent may actually turn to be irrelevant because of averaging effects, that is because for instance they are derivatives of small immaterial terms, or in the converse direction terms that are small may have huge derivatives and thus could yield strong effects. Besides, taking oscillating limits typically do no commute with nonlinear operations. This being stated we begin our formal process leading to an identification of possible asymptotic dynamics but try to use arguments as close as possible from those that we do justify in the linear setting in Section \ref{s:compactness}.
      
      To do so we pick a family of solutions $(f^\eps)$ and denote $(\phi^\eps)$ corresponding electric potentials (whether they actually depend on $\eps$ or not). We may normalize potentials by requiring for any $t$ and $\eps$
      \be\label{phi-normalized}
      \int_{\T^3}e^{-\sigma\,\phi^\eps(t,y)}\,\dD y\,=\,1\,.
      \ee
      This amounts to replacing the original $\phi^\eps$ with
      $$
      \phi^\eps\,-\,\ln\left(\int_{\T^3}e^{-\sigma\,\phi^\eps(t,y)}\,\dD y\right)\,,
      $$
      a process that as expected does not modify any equation.

      We stress that aiming at a unified in $\alpha$ derivation precludes any use of an asymptotic expansion in powers of $\eps$. Hence alternatively we rely directly on symmetries of the equations to sort out different terms of the equation. Incidentally note that even in an $\alpha$-by-$\alpha$ identification several symmetry arguments would be needed to derive that some terms that are expected to sum to zero actually vanish separately.
      
      \subsubsection*{Free energy identities} 
      To begin with, we introduce the time-dependent measure $\dD \mu_t^\eps=\cM^\eps(t,x,v)\,\dD x\,\dD v$ where the global Maxwellian is given by
      \be
      \cM^\eps(t,x,v)\,=\,e^{-\sigma\,\phi^\eps(t,x)}\,M(v)
      \label{BG}
      \ee
      and $M(v)=(2\pi)^{-3/2}\,e^{-|v|^2/2}$, so that, roughly speaking, transport terms and collision terms are respectively skew-symmetric and symmetric with respect to $\dD\mu^\eps_t$. As a consequence, setting
      $$
      h^\eps\,=\,f^\eps\,(\cM^\eps)^{-1},
      $$
      we obtain the following family of \textit{a priori} estimates.
      \bl Let $f^\eps$ be a smooth and localized solution of the magnetized Vlasov-Fokker-Planck equation \eqref{FPf} with smooth electric potential $\phi^\eps$. Then for any smooth function $H$,
      \be\label{e:H}
      \begin{array}{rcl}\ds
	\frac{\dD}{\dD t}\left(\iint_{\T^3\times\R^3} H(h^\eps)\,\dD\mu^\eps\right) 
	&+&\ds
	\frac{1}{\eps^{\alpha+1}}\iint_{\T^3\times\R^3} H''(h^\eps)\,|\nabla_v h^\eps|^2\dD\mu^\eps
	\\[0.5em]
	&=&\ds-\iint_{\T^3\times\R^3}(H'(h^\eps)h^\eps-H(h^\eps))\,\sigma\d_t\phi^\eps\,\dD\mu^\eps\,.
      \end{array}  
      \ee
      In particular,
      \be
      \frac{\dD}{\dD t}\left(\iint_{\T^3\times\R^3}f^\eps\left(\ln f^\eps + \frac{|v|^2}{2}+\sigma\phi^\eps\right) \right)
      +
      \frac{1}{\eps^{\alpha+1}}\iint_{\T^3\times\R^3} \frac{1}{h^\eps}\,|\nabla_v h^\eps|^2\dD\mu^\eps
      =\ds-\int_{\T^3}n^\eps\,\sigma\d_t\phi^\eps\,\,.
      \label{freeenerg}
      \ee
      \el
      \begin{proof} 
	We omit the detailed computation leading to~\eqref{e:H} as perfectly standard. Equation~\eqref{freeenerg} follows then from the choice $H(h)=h\,\ln(h)$.
      \end{proof}
      
      Assuming moreover \eqref{Poisson-electron} or \eqref{Poisson-ion} one may rewrite \eqref{freeenerg} in a form related to the conservation of some free energy. However since we shall make no use of such considerations we do not dwell on those here.
      
      Now we formally take limits in the free energy identities of the previous lemma. For the limiting functions, we drop the superscript $\eps$ (\textit{e.g.} $f$, $\phi$, $n$, $h$, $\mu$ ...). Since $\alpha>-1$, we get for all time $t\geq0$ 
      $$
      \int_0^t\iint_{\T^3\times\R^3} \frac{1}{h(s,\cdot,\cdot)}\,|\nabla_v h(s,\cdot,\cdot)|^2\dD\mu_s\quad\dD s\,=\,0.
      $$
      As a consequence $h$ only depends on space and time variables which means that the limiting distribution function is a local Maxwellian, namely $f = nM$. 
      By inserting this into equation~\eqref{FPf} 
      and balancing terms we can surmise\footnote{Unlike the rest of the formal analysis this expectation has no rigorous counterpart in our mathematical analysis. It is highly dubious that it could be fulfilled.} 
      that
      $$
      \tf^\eps=\frac{f^\eps-f}{\eps}
      $$
      should possess a limit $\tf$ that satisfies for any $(t,x,v)$
      $$
      v\cdot\nabla_xf(t,x,v)-\sigma\,\nabla_x\phi(t,x)\cdot\nabla_vf(t,x,v)-\sigma\,v^\perp\cdot\nabla_v\tf(t,x,v)\,=\,0\,.
      $$
      Now comes the key observation that the image of $v^\perp\cdot\nabla_v$ consists of functions with zero average with respect to the polar angle $\theta$ from cylinder coordinates $(|v_\perp|,\theta,v_\mypar)$ implicitly defined by $v=(|v_\perp|\cos(\theta),|v_\perp|\sin(\theta),v_\mypar)$. From this and the radial character of $f$ with respect to velocity variable stems 
      $$
      v_\mypar\,\d_{x_\mypar}f(t,x,v)-\sigma\,\d_{x_\mypar}\phi(t,x)\d_{v_\mypar}f(t,x,v)\,=\,0\,.
      $$
      Using once again that $f$ is Maxwellian in velocity we conclude that indeed
      $$
      f(t,x,v)\,=\,N(t,x_\perp)\,\frac{e^{-\sigma\phi(t,x)}}{\int_{\T} e^{-\sigma\phi(t,x_\perp,y_\mypar)}\,\dD\,y_\mypar}\,\frac{e^{-\frac{|v|^2}{2}}}{(2\pi)^{3/2}}
      $$
      where
      $$
      N(t,x_\perp)\,=\,\iint_{\T\times\R^3} f(t,x,v)\,\dD x_\mypar \dD v\,.
      $$
      
      \subsubsection*{Dynamics of velocity moments}
      The only task left is the identification of an equation for $N$. To proceed we denote by $N^\eps$ the integral of $f^\eps$ with respect to $(x_\mypar,v)$ and observe that a corresponding integration of \eqref{FPf} leads to a certain form of conservation of charge
      $$
      \d_t N^\eps\,+\,\textrm{div}_{x_\perp}(J^\eps_\perp)\,=\,0
      $$
      where
      $$
      J^\eps(t,x)\,=\,\frac{1}{\eps}\int_{\T\times\R^3} v\,f^\eps(t,x,v)\,\dD x_\mypar \dD v\,.
      $$
      This suggests that $J^\eps_\perp$ could possess a limit $J_\perp$. To identify this limit we proceed as above. By multiplying \eqref{FPf} by $v_\perp$ and integrating one derives
      $$
      \begin{array}{rcl}\ds
	\eps^2\,\d_tJ_\perp^\eps&+&\ds\textrm{div}_{x_\perp}\left(\int_{\T\times\R^3}v_\perp\otimes v_\perp\,f^\eps(\cdot,\cdot,x_\mypar,v)\dD x_\mypar \dD v\right)
	\ +\ \sigma\int_{\T}\nabla_{x_\perp}\phi^\eps(\cdot,\cdot,x_\mypar)\,n^\eps(\cdot,\cdot,x_\mypar)\dD x_\mypar\\[1em]
	&=&\ds
	-\sigma\,(J_\perp^\eps)^\perp\,-\,\eps^{1-\alpha}\,J_\perp^\eps
      \end{array}
      $$
      where ${}^\perp$ acts on vectors of $\R^2$ through $(j_1,j_2)^\perp=(-j_2,j_1)$. Taking limits brings
      $$
      -\sigma\,(J_\perp)^\perp\,=\,\textrm{div}_{x_\perp}\left(\int_{\T\times\R^3}v_\perp\otimes v_\perp\,f(\cdot,\cdot,x_\mypar,v)\dD x_\mypar \dD v\right)
      \ +\ \sigma\int_{\T}\nabla_{x_\perp}\phi(\cdot,\cdot,x_\mypar)\,n(\cdot,\cdot,x_\mypar)\dD x_\mypar
      $$
      which, by using the special form of $f$, may also be written as
      $$
      -\sigma\,(J_\perp)^\perp\,=\,\nabla_{x_\perp}N\,-\,N\,\nabla_{x_\perp}\left(\ln\left(\int_{\T}e^{-\sigma\phi(\cdot,\cdot,x_\mypar)}\,\dD x_\mypar\right)\right)\,.
      $$
      Inserting this in the limit of the foregoing equation of conservation of charge achieves the derivation of \eqref{guiding-center} after noticing that $(\nabla_{x_\perp}N)^\perp$ is divergence free.
      
      \br
      To ease comparisons we briefly sketch how in a collision-dominated scenario a similar formal analysis may be carried out. Therefore in the present remark we insert $\sigma_0\in\{0,1\}$ in front of the magnetic field term and allow $\sigma_0=0$ (non magnetized case) and $\alpha>1$ (very strong collisions). The condition $\alpha>-1$ is sufficient to support the expectation that $f^\eps$ converges to some $f=nM$. Then one turns to the analysis of moments
      $$
      n^\eps=\int_{\R^3}f^\eps(\cdot,\cdot,v)\dD v\,,\qquad\qquad
      j^\eps=\frac1\eps\int_{\R^3}v\,f^\eps(\cdot,\cdot,v)\dD v
      $$
      involved in the following version of conservation of charge
      $$
      \d_t n^\eps\,+\,\textrm{div}_{x}(j^\eps)\,=\,0
      $$
      and computes that
      \be\label{e:jeps}
      \eps^2\,\d_tj^\eps +
      \textrm{div}_{x}\left(\int_{\R^3}v\otimes v\,f^\eps\, \dD v\right)
      \ +\ \sigma\nabla_{x}\phi^\eps\,n^\eps
      \,=\,-\sigma\,\sigma_0\,(j^\eps)^\perp\,-\,\eps^{1-\alpha}\,j^\eps\,.
      \ee
      If $\sigma_0=0$ and $-1<\alpha<1$ taking the limit $\eps\to0$ and inserting $f=nM$ yields
      $$
      \nabla_x n\,+\,\sigma n\nabla_x\phi\,=\,0
      $$ 
      hence the classical Maxwell-Boltzmann approximation as for instance in \cite{bardos_2016_maxwell}. If $\alpha=1$ (no matter what $\sigma_0$ is) one rather obtains 
      $$
      \nabla_x n\,+\,\sigma n\nabla_x\phi
      \,=\,-\sigma\,\sigma_0\,j^\perp\,-\,j
      $$
      that may be solved for $j$ and inserted in $\d_tn+\text{div}_x(j)=0$ to derive a diffusive equation for $n$ as for instance in \cite{herda_2015_massless}. At last if $\alpha>1$ (no matter what $\sigma_0$ is) one derives $j=0$ hence $\d_tn=0$ thus the formal analysis suggests convergence to a global equilibrium. In our case we have essentially taken the limit of the third component of \eqref{e:jeps} to obtain Boltzmann-Gibbs behavior along the magnetic field and we have integrated the two first components of \eqref{e:jeps} to derive the equation for $N$.
      \label{BGwhy}
      \er
      
      \section{Well-posedness of asymptotic models}\label{s:WP}
      
      In this part, we investigate the well-posedness of the nonlinear asymptotic models and accordingly, we prove Theorem~\ref{th:strong} and Theorem~\ref{th:weak}. The asymptotic models for light and heavy species only differ by their Poisson-Boltzmann equation, which is 3D and anisotropic in the former case and 2D and isotropic in the latter. In Section~\ref{s:light}, we provide a detailed analysis of the anisotropic equation and the proof of the theorems for the light species case, the challenge being to prove 2D-like results on a model with some 3D features. In Section~\ref{s:heavy}, we indicate how to adapt the arguments to the simpler --- already 2D --- heavy species case.
      
      \subsection{Light particles}\label{s:light}
      
      In the present subsection we analyze the asymptotic system obtained in the electronic case so that, since it is consistent with modeling, for reading convenience we set $\sigma=-1$. Therefore we study the well-posedness of the following system
      \be
      \left\{
      \begin{aligned}
	&\d_t N\,+\,
	\textrm{div}_{x_\perp}(N\,\nabla_{x_\perp}^\perp\tphi)\,=\,0\,,\\
	&-\delta^2\Delta_x\phi\, =\, \un - N\,e^{\phi-\tphi}
	\,,\qquad \tphi(t,x_\perp)\,=\,\ln\left(\int_\T e^{\phi(t,x_\perp,x_\mypar)}\dD\,x_\mypar\right).
      \end{aligned}
      \right.
      \label{GCPB}
      \ee
      We are thus trying to solve a conservation law corresponding to advection by a divergence-free velocity field build from a "stream function" $\tphi$ obtained by averaging the solution to a second-order elliptic equation. As already pointed out in the introduction once sufficient control on the construction of $\tphi$ from $N$ has been obtained one may follow arguments used for similar equations such as Vlasov-Poisson equations \cite{Glassey,bouchut_2000_kinetic,Rein} or two-dimensional incompressible Euler equations in vorticity formulation \cite{Marchioro-Pulvirenti,Chemin,Majda-Bertozzi}. 
      
      \subsubsection*{Anisotropic Poisson-Boltzmann equation}
      
      Accordingly we first focus on the elliptic equation involved in System~\ref{GCPB}. Since time variable is a simple parameter there, we temporarily omit to mention time dependence along the following considerations. Moreover for concision's sake we do not repeat but always assume $N\geq0$, $\un\geq0$, $\int_{\T^3}\un=\int_{\T^2}N = 1$.
      
      Hence we study the following anisotropic Poisson-Boltzmann equation
      \be
      -\delta^2\Delta_x \phi(x_\perp, x_\mypar)\, =\, \un(x_\perp, x_\mypar) - N(x_\perp)\frac{e^{\phi(x_\perp, x_\mypar)}}{\int_\T e^{\phi(x_\perp, y_\mypar)}\dD y_\mypar}
      \label{PBani}
      \ee
      where  $N$ and $\un$ are nonnegative functions with integral equal to one. The key observation is that Equation~\ref{PBani} is the Euler-Lagrange equation associated to the energy functional 
      \be
      J[\psi]\, =\, \frac{1}{2}\delta^2\int_{\T^3}\left|\nabla\psi\right|^2
      + \int_{\T^2}N\, \ln\left(\int_\T e^{\psi}\, \dD x_\mypar\right)\,\dD x_\perp
      - \int_{\T^3} \un\,\psi\,.
      \label{energ}
      \ee
      Explicitly, at least for smooth pairs $(\phi,\psi)$, 
      \be
      \left.\frac{\dD}{\dD\tau}\right|_{\tau=0}J[\phi + \tau\psi]
      = \int_{\T^3}\left(\delta^2\nabla\phi\cdot\nabla\psi\ +\ N\,\frac{e^{\phi}}{\int_\T e^{\phi}\,\dD x_\mypar} \psi\ - \ \un\ \psi\ \right)\dD x\,
      \label{gateau}
      \ee
      We aim now at identifying an appropriate functional setting in which $J$ is strictly convex and coercive. From this shall stem existence and uniqueness for \eqref{PBani}. Afterwards we focus on corresponding regularity issues.
      
      To do so we introduce 
      \[
      H_0\,=\,\left\{\quad h\in\mathcal{D}'(\T^3)\quad|\quad\nabla h \in L^2(\T^3)\ \textrm{and}\ \int_{\T^3}h = 0\quad\right\}.
      \]
      The set $H_0$ is a closed linear subspace of the Sobolev space $H^1(\T^3)$ and by Mazur's theorem it is also weakly closed. 
      Moreover, by Sobolev embedding and Poincaré inequality $H_0\hookrightarrow L^6(\T^3)$.
      \bpr Assume $N\in L^{4/3}(\T^2)$ and $\un\in L^{6/5}(\T^3)$. Then 
      \begin{enumerate}
	\item[(i)] $J:H_0\rightarrow\R$ is well defined, bounded on bounded sets and coercive;
	\item[(ii)] $J$ is strictly convex;
	\item[(iii)] $J$ is Gâteaux differentiable;
	\item[(iv)] $J$ is weakly lower semi-continuous.
      \end{enumerate}
      \label{propJ}
      \epr
      \begin{proof}
	
	Since $H_0\hookrightarrow L^6(\T^3)$, for some $C$
	\[
	\int_{\T^3} \un |\psi|\dD x
	\,\leq\,\|\un\|_{L^{6/5}(\T^3)}\,\|\psi\|_{L^6(\T^3)}
	\,\leq\, C\,\|\un\|_{L^{6/5}(\T^3)}\,\|\psi\|_{H^1(\T^3)}.
	\]
	Likewise since\footnote{The last embedding follows from Gagliardo-Nirenberg inequality 
	  $$
	  \|\psi(x_\perp,\cdot)\|_{L^\infty(\T)}\leq C\,\|\psi(x_\perp,\cdot)\|_{L^6(\T)}^{3/4}\,
	  \|\psi(x_\perp,\cdot)\|_{H^1(\T)}^{1/4}
	  $$
	  combined with one H\"older inequality. From now on we shall make use of similar anisotropic Sobolev embeddings without further comment.} $H_0\hookrightarrow L^2(\T^2;H^1(\T))\cap L^6(\T^3)\hookrightarrow L^4(\T^2;L^\infty(\T))$, a pointwise estimate on the exponential function shows that
	\[
	\begin{array}{rcl}\ds
	  \int N\left|\ln\left(\int_\T e^{\psi(\cdot,y_\mypar)} \dD y_\mypar\right)\right|
	  &\leq&\ds
	  \|N\|_{L^{4/3}(\T^2)}
	  \left\|\ln\left(\int_\T e^{\psi(\cdot,y_\mypar)} \dD y_\mypar\right)\right\|_{L^4(\T^2)}\\[1em]
	  &\leq&\ds C\,
	  \|N\|_{L^{4/3}(\T^2)}
	  \|\psi\|_{L^4(\T^2;L^\infty(\T))}\\[1em]
	  &\leq&\ds C'\,
	  \|N\|_{L^{4/3}(\T^2)}
	  \|\psi\|_{L^6(\T^3)}^{3/4}
	  \|\psi\|_{L^2(\T^2;H^1(\T))}^{1/4}\\[1em]
	  &\leq&\ds C''\,
	  \|N\|_{L^{4/3}(\T^2)}
	  \|\psi\|_{H^1(\T^3)}
	\end{array}
	\]
	for some $C,C',C''$. Combined with a Poincar\'e inequality this provides 
	\be
	\left|J[\psi] -\frac{1}{2}\delta^2\|\nabla\psi\|^2_{L^2(\T)^3}\right|\leq\, 
	C(\|N\|_{L^{4/3}(\T^2)} + \|\un\|_{L^{6/5}(\T^3)})\|\nabla\psi\|_{L^2(\T^3)},
	\label{coerc}
	\ee
	for some $C$, which proves the first claim.
	
	Thanks to the Hölder inequality it holds, for $\theta\in[0,1]$
	\[
	\int_\T e^{\theta\psi_1 +(1-\theta)\psi_2} \dD x_\mypar\leq \left(\int_\T e^{\psi_1} \dD x_\mypar\right)^\theta\left(\int_\T e^{\psi_2} \dD x_\mypar\right)^{1-\theta}.
	\]
	Therefore the second term of the right-hand side of \eqref{energ} is convex. The first is strictly convex and the last is linear. This proves the second claim.
	
	To prove differentiability, since the first and third term of $J$ are respectively quadratic and linear continuous, one may safely focus on the middle term of $J$. The latter differentiability follows from applying twice the Dominated Convergence Theorem. To do so, we note that when $(\phi,\psi)\in (H_0)^2$ and $(t,s)\in\R^2$, for almost any $x_\perp$
	$$
	\begin{array}{rcl}\ds
	  \Big|\ln\left(\int_\T e^{(\phi+t\psi)(x_\perp,y_\mypar)} \dD y_\mypar\right)
	  &-&\ds
	  \ln\left(\int_\T e^{(\phi+s\psi)(x_\perp,z_\mypar)} \dD z_\mypar\right)\Big|\\[1em]
	  &=&\ds
	  \ln\left(\int_\T e^{(t-s)\psi(x_\perp,y_\mypar)}\,\frac{e^{(\phi+s\psi)(x_\perp,y_\mypar)}\dD y_\mypar}{\int_\T e^{(\phi+s\psi)(x_\perp,z_\mypar)} \dD z_\mypar}\right)\\[1em]
	  &\leq&\ds
	  \ln\left(\int_\T e^{|t-s|\,\|\psi(x_\perp,\cdot)\|_{L^\infty(\T)}}\,\frac{e^{(\phi+s\psi)(x_\perp,y_\mypar)}\dD y_\mypar}{\int_\T e^{(\phi+s\psi)(x_\perp,z_\mypar)} \dD z_\mypar}\right)\\[1em]
	  &\leq&\ds |t-s|\,\|\psi(x_\perp,\cdot)\|_{L^\infty(\T)}
	\end{array}
	$$
	where we have broken the symmetry by assuming that 
	$$
	\int_\T e^{(\phi+s\psi)(x_\perp,z_\mypar)} \dD z_\mypar \leq \int_\T e^{(\phi+t\psi)(x_\perp,y_\mypar)} \dD y_\mypar
	$$
	and used monotonicity of $\ln$ and $\exp$. From this follows the third claim.
	
	The last claim follows for instance from convexity and local boundedness combined with Mazur's theorem.
      \end{proof}
      
      The foregoing proposition yields by classical arguments the following one.
      
      \bpr[Existence and uniqueness]
      For any $N\in L^{4/3}(\T^2)$ and $\un\in L^{6/5}(\T^3)$, Equation~\eqref{PBani} possesses a unique weak solution $\phi\in H_0$. Moreover there exists $C$ such that for any such $(N,\un)$ the solution $\phi$ satisfies
      $$
      \|\phi\|_{H^1(\T^3)}\,\leq\,
      C(\|N\|_{L^{4/3}(\T^2)} + \|\un\|_{L^{6/5}(\T^3)})\,.
      $$
      Besides there exists $C$ such that for any pair $(N,\un)$, $(N',\un')$ of such couples respective solutions $\phi$ and $\phi'$ satisfy
      $$
      \|\phi-\phi'\|_{H^1(\T^3)}\,\leq\,
      C(\|N-N'\|_{L^{4/3}(\T^2)} + \|\un-\un'\|_{L^{6/5}(\T^3)})
      $$
      \label{existPB}
      \epr
      \begin{proof}
	Existence and uniqueness follows from Proposition \ref{propJ} along the classical line of the direct method of calculus of variations. Then the uniform bound stems directly from coercivity estimate \eqref{coerc} and the fact that $J$ takes the value $0$ at the null function. At last, as we explain now, the Lipschitz estimate follows from quantifying convexity and local boundedness. Temporarily, to make it precise, we mark $J$ with suffixes ${N,\un}$. By using that $\phi$ and $\phi'$ are critical points of, respectively, $J_{N,\un}$ and $J_{N',\un'}$ and that each of the three terms defining $J$ is convex, one obtains 
	$$
	\begin{array}{rcl}\ds
	  \delta^2\,\|\nabla(\phi-\phi')\|_{L^2(\T^3)}^2
	  &\leq&\ds 
	  J_{N,\un}(\phi')-J_{N,\un}(\phi)+J_{N',\un'}(\phi)-J_{N',\un'}(\phi')\\[0.5em]
	  &=&\ds
	  \int_{\T^2}(N-N')\, \ln\left(\frac{\int_\T e^{\phi'}\, \dD x_\mypar}{\int_\T e^{\phi}\, \dD x_\mypar}\right)\,\dD x_\perp
	  - \int_{\T^3}(\un-\un')\,(\phi'-\phi)\,\dD x\\[0.5em]
	  &\leq&\ds
	  C'(\|N-N'\|_{L^{4/3}(\T^2)} + \|\un-\un'\|_{L^{6/5}(\T^3)})\,\|\phi-\phi'\|_{H^1(\T^3)}
	\end{array}
	$$
	for some $C'$, where we have used the pointwise bound
	\be
	\ln\left(\frac{\int_\T e^{\phi'(x_\perp,x_\mypar)}\, \dD x_\mypar}{\int_\T e^{\phi(x_\perp,x_\mypar)}\, \dD x_\mypar}\right)
	\,\leq\,\|(\phi-\phi')(x_\perp,\cdot)\|_{L^\infty(\T)}
	\label{pointwisebound}
	\ee
	and the embedding $H_0\hookrightarrow L^4(\T^2;L^\infty(\T))$. From here a Poincar\'e inequality yields the result.
      \end{proof}
      
      Now we discuss regularity of solutions to Equation~\eqref{PBani}. In the next proposition we use classical elliptic regularity properties --- in Calder\'on-Zygmund form --- and maximum principles --- as recalled in Appendix~\ref{s:maxprinc} --- to obtain estimates of $\phi$ in higher order Sobolev spaces. We refer the reader to \cite{Stein_singular-integrals,Stein-Weiss_Fourier} or \cite{Grafakos_I} on the classical Calder\'on-Zygmund regularity theory and to \cite[Chapter~2]{Lemarie-Rieusset} for relevant basic properties of Lorentz spaces appearing in the following.
      \bpr[Regularity]
      There exists $C$ such that if $N\in L^{4/3}(\TT^2)$ and $\un\in L^{3/2,1}(\TT^3)$ then $\phi$ the unique solution to \eqref{PBani} satisfies
      $$
      \|e^\phi\|_{L^\infty(\T^3)}\,+\,\left\|\frac{1}{\int_\T e^{\phi(\cdot, y_\mypar)}\dD y_\mypar}\right\|_{L^\infty(\T^2)}\,\leq\,C\,e^{C\left(\|\un\|_{L^{3/2,1}(\TT^3)}+\|N\|_{L^{4/3}(\TT^2)}\right)}
      $$
      and if moreover $N\in L^{3/2,1}(\TT^2)$
      $$
      \|\phi\|_{L^\infty(\T^3)}\leq 
      C\,\left(\|\un\|_{L^{3/2,1}(\TT^3)}+\|N\|_{L^{3/2,1}(\TT^2)}e^{C\left(\|\un\|_{L^{3/2,1}(\TT^3)}+\|N\|_{L^{4/3}(\TT^2)}\right)}\right)\,.
      $$
      Moreover, for any $1<p<\infty$, there exists $C$ such that if $N\in L^{4/3}(\TT^2)\cap L^p(\TT^2)$ and $\un\in L^{3/2,1}(\TT^3)\cap L^p(\TT^3)$ then $\phi$ the unique solution to \eqref{PBani} satisfies
      $$
      \|\phi\|_{W^{2,p}(\T^3)}\leq 
      C\,\left(\|\un\|_{L^{p}(\TT^3)}+\|N\|_{L^{p}(\TT^2)}e^{C\left(\|\un\|_{L^{3/2,1}(\TT^3)}+\|N\|_{L^{4/3}(\TT^2)}\right)}\right)\,.
      $$
      Besides for any $1<p<\infty$ and any $1\leq q\leq\infty$ 
      such that 
      $$
      \frac{3}{2p}-1<\frac1q\leq\frac1p
      $$
      there exists $C$ such that with 
      $$
      \frac1r\,=\,
      \begin{cases}
	\frac13+\frac1p-\frac{2}{3q}&\quad\textrm{if }\quad q\geq2\\
	\frac12+\frac1p-\frac1q&\quad\textrm{if }\quad q\leq2
      \end{cases}
      $$
      if $N\in L^q(\TT^2)\cap L^r(\TT^2)\cap W^{1,p}(\TT^2)$ and $\un\in L^r(\TT^3)\cap W^{1,p}(\TT^3)$
      then $\phi$ the unique solution to \eqref{PBani} satisfies
      $$
      \begin{array}{rcl}\ds
	\|\phi\|_{W^{3,p}(\T^3)}&\leq&\ds
	C\,\left(\|\un\|_{W^{1,p}(\TT^3)}+\|N\|_{L^{q}(\TT^2)}\|\un\|_{L^{r}(\TT^3)}e^{C\left(\|\un\|_{L^{3/2,1}(\TT^3)}+\|N\|_{L^{4/3}(\TT^2)}\right)}\right)\\[0.5em]
	&+&\ds
	C\,e^{C\left(\|\un\|_{L^{3/2,1}(\TT^3)}+\|N\|_{L^{4/3}(\TT^2)}\right)}
	\left(\|N\|_{W^{1,p}(\TT^2)}+\|N\|_{L^{q}(\TT^2)}\|N\|_{L^{r}(\TT^2)}\right)\,.
      \end{array}
      $$
      \label{regellip}
      \epr
      
      To ease later use of estimates of $\|\phi\|_{W^{3,p}(\T^3)}$ note that for instance the choice $$\frac1q\,=\,\frac12\left(\frac1p+\frac{1}{\max(\{2,p\})}\right)$$ 
      is always available and that in this case $r\leq q$.
      
      \begin{proof}
	Bounds in $L^\infty$ follow from repeated use of the maximum principle stated in Lemma \ref{maxprinc}. Indeed, first, since $L^{3/2,1}(\TT^3)\hookrightarrow L^{6/5}(\TT^3)$ and $-\delta^2\Delta\phi\leq \un$, the bound on $e^\phi$ follows from $0\leq e^{\phi}\leq e^{\esssup \phi}$. In turn the bound on $(\int_\T e^{\phi(x_\perp, y_\mypar)}\dD y_\mypar)^{-1}$ stems from
	$$
	0\leq \frac{1}{\int_\T e^{\phi(x_\perp, y_\mypar)}\dD y_\mypar} \leq 
	e^{-\int_\T \phi(x_\perp, y_\mypar)\dD y_\mypar} \leq
	e^{\esssup(-\int_\T \phi(x_\perp, y_\mypar)\dD y_\mypar)}
	$$
	and $-\Delta_{x_\perp} (-\int_\T \phi\dD x_\mypar)= N-\int_\T \un(\cdot, y_\mypar)\dD y_\mypar$ combined with embeddings $W^{2,4/3}(\T^2)\hookrightarrow L^{\infty}(\TT^2)$ and $L^{3/2,1}(\TT^3)\hookrightarrow L^{4/3}(\TT^2;L^1(\T))$. At last, using the equation, those two estimates yield then estimates on $\phi$ itself both in $L^\infty(\T^3)$ and $W^{2,p}(\T^3)$ respectively by maximum principle and elliptic regularity.
	
	The bound in $W^{3,p}(\T^3)$ may then be obtained by differentiating once the equation. Indeed differentiation yields
	\[
	-\delta^2\Delta\nabla\phi\,=\,\nabla\un - \left(\nabla N + N\nabla\phi - N\nabla\tphi\right)\,e^{\phi-\tphi}
	\]
	where we have introduced notation $\tphi$ defined through nonlinear averaging~\eqref{averaged-phi} for concision's sake. Moreover, by using Lemma~\ref{l:tphi} below to estimate $\tphi$ and H\"older inequalities, we obtain for some $C$ independent of $N$
	\[
	\|N\nabla\phi\|_{L^p(\T^3)} + \|N\nabla\tphi\|_{L^p(\T^3)}\,\leq\,Ce^{C\left(\|\un\|_{L^{3/2,1}(\TT^3)}+\|N\|_{L^{4/3}(\TT^2)}\right)}\|N\|_{L^q(\T^2)}\|\nabla\phi\|_{L^s(\T^2;L^p(\T))}.
	\]
	with $1/s=1/p-1/q$. One achieves the proof of the claim through elliptic regularity and bounds on $\|\phi\|_{W^{2,r}(\T^3)}$ combined with embedding $W^{1,r}(\TT^3)\hookrightarrow L^{s}(\TT^2;L^p(\T))$. Strictly speaking instead of differentiating the equation and estimating as above terms thus appearing we should have applied finite differences to the equation and estimated those finite differences. But here and elsewhere we skip over those cumbersome details. This is particularly immaterial in our analysis since there is no boundary to deal with.
      \end{proof}
      
      With the obtained regularity one may provide higher-order counter-parts to the Lipschitz bound of Proposition~\ref{existPB}.

      \bpr[Lipschitz dependence]
      For any $4/3\leq p<\infty$, there exists $C$ such that if $N\in L^p(\TT^2)$, $N'\in L^p(\TT^2)$ and $\un\in L^{3/2,1}(\TT^3)$ then $\phi$ and $\phi'$ the respective solutions to \eqref{PBani} corresponding to $(N,\un)$ and $(N',\un)$ satisfy
      $$
      \|\phi-\phi'\|_{W^{2,p}(\T^3)}
      \leq C\,\|N-N'\|_{L^p(\T^2)}\,(1+\|N'\|_{L^p(\T^2)})^2e^{C\left(\|\un\|_{L^{3/2,1}(\TT^3)}+\|N\|_{L^{4/3}(\TT^2)}+\|N'\|_{L^{4/3}(\TT^2)}\right)}\,.
      $$
      \label{Lipschitz-dependence}
      \epr
      
      \begin{proof}
	The power $2$ in the above estimate reflects the fact that we provide a three-steps proof to reach the full range of Lebesgue indices $p$. First observe that $H^1(\TT^3)\hookrightarrow L^q(\TT^2;L^2(\T))$ for any $1\leq q<\infty$. Using the equation, we get
	\[
	-\delta^2\Delta(\phi-\phi')\,=\,(N-N')e^{\phi-\tphi} + (e^\phi-e^{\phi'})N'e^{-\tphi} + (e^{-\tphi}-e^{-\tphi'})N'e^{\phi'}
	\]   
	where again we use notation from \eqref{averaged-phi}. By using Lemma~\ref{l:tphi} this enables us to prove that for any $1\leq r<4/3$ there exists $C_r$ such that 
	$$
	\|\phi-\phi'\|_{W^{2,r}(\T^3)}
	\leq C_r\,\|N-N'\|_{L^{4/3}(\T^2)}\,e^{C_r\left(\|\un\|_{L^{3/2,1}(\TT^3)}+\|N\|_{L^{4/3}(\TT^2)}+\|N'\|_{L^{4/3}(\TT^2)}\right)}\,.
	$$
	
	Using once again the equation and the embedding $W^{2,r_p}(\TT^3)\hookrightarrow L^\infty(\TT^2;L^p(\T))$ for a suitable $1\leq r_p<4/3$, one deduces the result (with a bound depending linearly on $\|N'\|_{L^p(\T^2)}$) when $p<4$.
	
	One may then use this intermediate case to conclude the proof by using again the equation and the embedding $W^{2,p_0}(\TT^3)\hookrightarrow L^\infty(\TT^3)$ for some $1\leq p_0<4$ when $p\geq4$.
      \end{proof}
      
      Prior to turning back to Equation~\eqref{GCPB}, we provide estimates converting bounds on a potential $\phi$ into bounds on advection field $\tphi$ obtained through nonlinear averaging~\eqref{averaged-phi}. That is the content of the following lemma. It is of course necessary to deal with \eqref{GCPB} and we have already used it repeatedly in foregoing estimates on $\phi$. 
      
      \bl\label{l:tphi}
      \emph{(i).} For any $\phi\in H_0$, let $\tphi$ be defined by \eqref{averaged-phi}, then for any $1\leq p\leq\infty$
      \[
      \|\tphi\|_{L^p(\T^2)}\ \leq\ \|\phi\|_{L^p(\T^2;L^\infty(\T))}
      \]
      and
      \[
      \|\nabla_{x_\perp}\tphi\|_{L^p(\T^2)}\leq 
      \min(\{\|\nabla_{x_\perp}\phi\|_{L^p(\T^2;L^\infty(\T))},
      \|e^{\phi}\|_{L^\infty(\T^3)}\|e^{-\tphi}\|_{L^\infty(\T^2)}\|\nabla_{x_\perp}\phi\|_{L^p(\T^2;L^1(\T))}\})\,.
      \]
      Moreover there exists $C$ such that for any $\phi$, $p$ and $\tphi$ as above,
      $$
      \|\mathrm{Hess}_{x_\perp}\tphi\|_{L^p(\T^2)}\leq C\,\|e^\phi\|_{L^\infty(\T^3)}^2\|e^{-\tphi}\|_{L^\infty(\T^2)}^2\,
      \left(\|\mathrm{Hess}_{x_\perp}\phi\|_{L^p(\T^2; L^1(\T))} + \|\nabla_{x_\perp}\phi\|_{L^{2p}(\T^2;L^2(\T))}^2\right).
      $$
      \emph{(ii).} For any $\phi,\phi'\in H_0$, let $\tphi$ and $\tphi'$ be defined by nonlinear averaging \eqref{averaged-phi} respectively from $\phi$ and $\phi'$, then for any $1\leq p\leq\infty$
      $$
      \begin{array}{rcl}\ds
	\|\tphi-\tphi'\|_{L^p(\T^2)}&\leq&\ds
	\min(\{\|\phi-\phi'\|_{L^p(\T^2;L^\infty(\T))},\\
	&&\ds
	(\|e^\phi\|_{L^\infty(\T^3)}+ \|e^{\phi'}\|_{L^\infty(\T^3)})
	(\|e^{-\tphi}\|_{L^\infty(\T^2)}+ \|e^{-\tphi'}\|_{L^\infty(\T^2)})\|\phi-\phi'\|_{L^p(\T^2;L^1(\T))}\})
      \end{array}
      $$
      and
      $$
      \begin{array}{rcl}\ds
	\hspace{-1em}\|&\hspace{-1em}\nabla_{x_\perp}&\hspace{-1em}(\tphi-\tphi')\|_{L^p(\T^2)}\\[0.5em] 
	&\leq&\ds
	\min(\{\|\nabla_{x_\perp}(\phi-\phi')\|_{L^p(\T^2;L^\infty(\T))},
	\|e^{\phi}\|_{L^\infty(\T^3)}\|e^{-\tphi}\|_{L^\infty(\T^2)}\|\nabla_{x_\perp}(\phi-\phi')\|_{L^p(\T^2;L^1(\T))}\})\\[0.5em]
	&\hspace{-1em}+\hspace{-1em}&\ds\hspace{-1em}(\|e^\phi\|_{L^\infty(\T^3)}+ \|e^{\phi'}\|_{L^\infty(\T^3)})
	(\|e^{-\tphi}\|_{L^\infty(\T^2)}+ \|e^{-\tphi'}\|_{L^\infty(\T^2)})
	\|\nabla_{x_\perp}\phi'\|_{L^p(\T^2;L^1(\T))}\|\phi-\phi'\|_{L^p(\T^2;L^\infty(\T))}\,.
      \end{array}
      $$
      \el
      
      To ease subsequent use of the lemma we stress that
      \begin{itemize}
	\item $W^{2,p}(\T^3)\hookrightarrow W^{1,2p}(\T^2;L^2(\T))$ provided that $p\geq 4/3$;
	\item $W^{3,q}(\T^3)\hookrightarrow W^{2,p}(\T^2;L^1(\T))\cap W^{1,2p}(\T^2;L^2(\T))$ provided that $1/q\leq 1/2+1/p$ when $p<\infty$ and that $1/q<1/2$ when $p=\infty$;
	\item $W^{2,q}(\T^3)\hookrightarrow L^p(\T^2;L^{\infty}(\T))\cap W^{1,p}(\T^2;L^1(\T))$ provided that $1/q\leq 1/2+1/p$ when $p<\infty$ and that $1/q<1/2$ when $p=\infty$.
      \end{itemize}
      
      \begin{proof}
	Since  $H^1(\T^3)\hookrightarrow L^4(\T^2; L^\infty(\T))$, $\tphi$ is well-defined. Moreover, the first estimate follows directly from a pointwise bound and the second from the triangle inequality for the $L^p(\T^2)$-norm. Concerning the latter note that indeed 
	\[
	\nabla_{x_\perp}\tphi(x_\perp)\, =\, e^{-\tphi(x_\perp)}\,
	\int_\T e^{\phi(x_\perp,y_\mypar)}\,\nabla_{x_\perp}\phi(x_\perp,y_\mypar)\ \dD y_\mypar\,.
	\]
	Likewise
	$$
	\begin{array}{rcl}\ds
	  \text{Hess}_{x_\perp}(\tphi)(x_\perp)
	  &=&\ds
	  e^{-\tphi(x_\perp)}\,
	  \int_\T e^{\phi(x_\perp,y_\mypar)}\,(\text{Hess}_{x_\perp}(\phi)+\nabla_{x_\perp}\phi\otimes\nabla_{x_\perp}\phi)(x_\perp,y_\mypar)\ \dD y_\mypar\\[0.5em]
	  &-&\ds
	  e^{-2\tphi(x_\perp)}\,
	  \iint_{\T\times\T} e^{\phi(x_\perp,y_\mypar)+\phi(x_\perp,z_\mypar)}\,\nabla_{x_\perp}\phi(x_\perp,y_\mypar)\otimes\nabla_{x_\perp}\phi(x_\perp,z_\mypar)\ \dD y_\mypar\dD z_\mypar\,.
	\end{array}
	$$
	implies the third inequality. This achieves the proof of (i).
	
	Concerning the first inequality of (ii), the $L^p(\T^2;L^\infty(\T))$ estimate follows directly from \eqref{pointwisebound}. To prove the second part of this first inequality we first observe that, breaking the symmetry as in the proof of \eqref{pointwisebound} by assuming $\tphi'(x_\perp)\leq\tphi(x_\perp)$, one obtains
	$$
	\begin{array}{rcl}\ds
	  |\tphi(x_\perp)-\tphi'(x_\perp)|
	  &\leq&\ds
	  \ln\left(1+e^{-\tphi(x_\perp)}\int_\T ((\phi-\phi')e^{\phi'})(x_\perp,x_\mypar)\, \dD x_\mypar\right)\\
	  &\leq&\ds\|e^{\phi'}\|_{L^\infty(\T^3)}\|e^{-\tphi}\|_{L^\infty(\T^2)}
	  \|(\phi-\phi')(x_\perp,\cdot)\|_{L^1(\T)}
	\end{array}
	$$
	by using both $(\forall z\in\R,\quad e^z\leq 1+ze^z)$ and $(\forall x\in\R_+^*,\quad \ln(x)\leq x-1)$. Integrating a symmetrized form of the foregoing inequality achieves the proof of the first inequality of (ii). The second one follows almost readily from
	$$
	\begin{array}{rcl}\ds
	  \nabla_{x_\perp}\tphi(x_\perp)-\nabla_{x_\perp}\tphi'(x_\perp)&=&\ds
	  e^{-\tphi(x_\perp)}\,
	  \int_\T e^{\phi(x_\perp,y_\mypar)}\,\nabla_{x_\perp}(\phi-\phi')(x_\perp,y_\mypar)\ \dD y_\mypar\\
	  &+&\ds
	  (e^{-\tphi(x_\perp)}-e^{-\tphi'(x_\perp)})\,
	  \int_\T e^{\phi(x_\perp,y_\mypar)}\,\nabla_{x_\perp}\phi'(x_\perp,y_\mypar)\ \dD y_\mypar\\
	  &+&\ds
	  e^{-\tphi'(x_\perp)}\,
	  \int_\T (e^{\phi(x_\perp,y_\mypar)}-e^{\phi'(x_\perp,y_\mypar)})\,\nabla_{x_\perp}\phi'(x_\perp,y_\mypar)\ \dD y_\mypar\,.
	\end{array}
	$$
      \end{proof}
      
      \subsubsection*{Existence of solutions}
      
      \begin{proof}[Proof of Theorem \ref{th:strong} for light particles] 
	The strong framework allows to work with Lipschitz fields $\nabla^\perp \tphi$ and propagate $W^{1,p}(\T^2)$ bounds. The other key observation is that it is also sufficient to prove a Lipschitz bound in low-order norms. For instance, with constants depending only on bounds on $\|N\|_{L^{4/3}(\T^2)}$, $\|N'\|_{L^{4/3}(\T^2)}$ and $\|\un\|_{L^{3/2,1}(\TT^3)}$
	$$
	\begin{array}{rcl}\ds
	  \|(\nabla^\perp\tphi-\nabla^\perp\tphi')\cdot\nabla N'\|_{L^{4/3}(\T^2)}
	  &\leq&\ds
	  C \|\nabla(\tphi-\tphi')\|_{L^4(\T^2)}\|\nabla N'\|_{L^2(\T^2)}\\
	  &\leq&\ds
	  C'\|\phi-\phi'\|_{W^{2,4/3}(\T^3)}\|\nabla N'\|_{L^2(\T^2)}\\
	  &\leq&\ds
	  C''\|N-N'\|_{L^{4/3}(\T^2)}\|\nabla N'\|_{L^2(\T^2)}\,
	\end{array}
	$$
	using Proposition \ref{Lipschitz-dependence} and Lemma \ref{l:tphi} and since $W^{2,4/3}(\T^3)\hookrightarrow W^{1,4}(\T^2;L^1(\T))\cap L^4(\T^2;L^\infty(\T))$. This is sufficient to prove uniqueness (even of weak-strong type) and convergence of existence schemes.
	
	As it is fairly classical we only sketch the main steps of the proof. The uniqueness follows for instance from the fact that if $N$ and $N'$ are two solutions starting with the same initial data then on any compact interval $I_0$ containing zero on which both $N$ and $N'$ are defined
	$$
	\d_t (N-N')+\nabla^\perp\tphi\cdot\nabla(N-N')=-(\nabla^\perp\tphi-\nabla^\perp\tphi')\cdot\nabla N'
	$$
	so that for any $t\in I_0$
	$$
	\|(N-N')(t,\cdot)\|_{L^{4/3}(\T^2)}\leq C\left|\int_0^t \|N(s,\cdot)-N'(s,\cdot)\|_{L^{4/3}(\T^2)}\dD s\right|
	$$
	for some $C$ depending on $N$ and $N'$. Then $N=N'$ follows from the Gronwall lemma. Existence on an interval of controlled length may be proved by showing convergence of the scheme
	$$
	\d_t N_{n+1}+\nabla^\perp\tphi_n\cdot\nabla N_{n+1}\,=\,0\,,\quad n\geq0\,,
	$$ 
	where $\tphi_{-1}=0$, with $N_{n+1}(0,\cdot)=N^\text{in}$. On any interval of time $I$ containing zero the divergence-free structure of the vector field yields bounds on $(N_n)$ in $L^\infty (I;L^\infty(\T^2))$, independent of $I$, by $\|N^\text{in}\|_{L^\infty(\T^2)}$. Then, on such an interval, by differentiating the equation and using Lemma~\ref{l:tphi} and Proposition~\ref{regellip} one obtains (with obvious modification when $n=0$) for any $t\in I$
	\[
	\begin{array}{rcl}\ds
	  \|\nabla_{x_\perp}N_{n+1}\|_{L^p(\T^2)}
	  &\leq&\ds
	  \|\nabla_{x_\perp}N^\text{in}\|_{L^p(\T^2)}
	  +\left|\int_0^t\|\tphi_n(s,\cdot)\|_{W^{2,\infty}(\T^2)}\|\nabla_{x_\perp}N_{n+1}(s,\cdot)\|_{L^p(\T^2)}\,\dD s\right|\\
	  &\leq&\ds
	  \|\nabla_{x_\perp}N^\text{in}\|_{L^p(\T^2)}
	  +C\left|\int_0^t\|N_n(s,\cdot)\|_{W^{1,p}(\T^2)}^2\|\nabla_{x_\perp}N_{n+1}(s,\cdot)\|_{L^p(\T^2)}\,\dD s\right|.
	\end{array}
	\]
	provided that $2<p<\infty$ and for some $C$ depending only on $p$, $\|N^\text{in}\|_{L^\infty(\T^2)}$ and $\|\un\|_{L^\infty(I;W^{1,p}(\T^3))}$. Therefore, on a suitably short interval $I$ one may also obtain uniform bounds of $(N_n)$ in $L^\infty (I;W^{1,p}(\T^2))$ where both the bound on $(N_n)$ and the smallness of $I$ are only constrained by the sizes of $\|N^\text{in}\|_{W^{1,p}(\T^2)}$ and of $\|\un\|_{L^\infty(I;W^{1,p}(\T^3))}$. Now, by already expounded arguments, on the same intervals one may prove convergence of $(N_n)$ in $\cC(I;L^{4/3}(\T^2))$ to some $N$. One then concludes the proof of the claimed existence by upgrading this convergence through interpolation with uniform bounds from Proposition \ref{regellip} and taking limits in the equation using Lipschitz bounds on maps $N\mapsto\phi$ and $\phi\mapsto\tphi$ from Proposition \ref{Lipschitz-dependence} and Lemma \ref{l:tphi}. At last one achieves the proof by combining the strong uniqueness statement and existence on a controlled interval with classical arguments leading to the existence of a maximal solution.
      \end{proof}

      \begin{proof}[Proof of Theorem \ref{th:weak} for light particles] 
	While a suitable Lipschitz bound is the key step in proving Theorem~\ref{th:strong}, here the main issue is the stability of some approximate solutions $(N_\eps)$ and their compactness in compatible norms. With this respect, one key observation is that in the Lipschitz bound of Proposition~\ref{existPB} one may replace the $L^{4/3}$-norm with a $W^{-1/2,2}$-norm. This follows from
	$$
	\begin{array}{rcl}\ds
	  \Big|\ln\left(\frac{\int_\T e^{\phi'(x_\perp,x_\mypar)}\, \dD x_\mypar}{\int_\T e^{\phi(x_\perp,x_\mypar)}\, \dD x_\mypar}\right)
	  &-&\ds\ln\left(\frac{\int_\T e^{\phi'(x_\perp',x_\mypar)}\, \dD x_\mypar}{\int_\T e^{\phi(x_\perp',x_\mypar)}\, \dD x_\mypar}\right)\Big|\\[0.5em]
	  &\leq&\ds
	  \left|\|(\phi-\phi')(x_\perp,\cdot)\|_{L^\infty(\T)}-\|(\phi-\phi')(x_\perp',\cdot)\|_{L^\infty(\T)}\right|
	\end{array}
	$$
	and the embedding $H_0\hookrightarrow W^{1,2}(\T^2;L^2(\T))\cap L^2(\T^2;W^{1,2}(\T))\hookrightarrow W^{1/2,2}(\T^2;L^\infty(\T))$ that yield
	$$
	\left\|\ln\left(\frac{\int_\T e^{\phi'(\cdot,x_\mypar)}\, \dD x_\mypar}{\int_\T e^{\phi(\cdot,x_\mypar)}\, \dD x_\mypar}\right)\right\|_{W^{1/2,2}(\T^2)}
	\leq C\,\|\phi-\phi'\|_{H^1(\T^3)}
	$$
	for some universal $C$.
	
	To be more specific, let us discuss how to prove existence by a compactness argument on the family of solutions $(N_\eta)$ to
	$$
	\d_t N_\eta+\nabla^\perp\tphi_\eta\cdot\nabla N_\eta\,=\,0\,,\quad \eta>0\,,
	$$
	starting from $N^\text{in}$ where $\tphi_\eta$ is obtained from $N_\eta$ through
	$$
	-\delta^2\Delta_x\phi_\eta\, =\, \un - N_\eta\frac{e^{\phi_\eta}}{\int_\T e^{\phi_\eta(\cdot,\cdot,y_\mypar)}\dD y_\mypar}
	\,,\qquad \tphi_\eta(t,x_\perp)\,=\,\zeta_\eta\,*\,\ln\left(\int_\T e^{\phi_\eta(t,x_\perp,x_\mypar)}\dD\,x_\mypar\right)
	$$
	where $(\zeta_\eta)$ is an approximation of unity. The existence of $(N_\eta)$ may be obtained by an argument similar but much simpler than the one leading to Theorem~\ref{th:strong} with the crucial modification that one does not need to restrict the interval of existence to obtain needed bounds (which are of course $\eta$-dependent). 
	
	Now, the divergence-free property of vector fields $(\nabla^\perp\tphi_\eta)$ provide control of $(N_\eta)$ in $L^\infty(\R;L^p(\T^2))$ uniformly in $\eta$. Then bounds in $W_{loc}^{1,\infty}(\R;W^{-1,q}(\T^2))$ with $1/q = 2/p-1/2$ if $p<2$, any $1\leq q<2$ if $p=2$ and $q=p$ if $p>2$ may be derived directly from the equation in conservative form that implies for any compact interval $I$ and any $t\in I$
	$$
	\|\d_tN_\eta(t,\cdot)\|_{W^{-1,q}(\T^2)}\leq
	\|\nabla_{x_\perp}\tphi_\eta(t,\cdot)\|_{L^r(\T^2)}\|N_\eta(t,\cdot)\|_{L^p(\T^2)}
	\leq C\,.
	$$
	with $1/r=1/q-1/p$ for some constant $C$ depending only on $\|N^\text{in}\|_{L^p(\T^2)}$, $\|\un\|_{L^\infty(I;L^p(\T^3)\cap L^{3/2,1}(\T^3))}$ (and $q$ if $p=2$). Therefore, compactness of $(N_\eta)$ in $\cC(\R_+;W^{-1/2,2}(\T^2))$ may be obtained by interpolation between boundedness in $L^p(\T^2)$ and Lipschitz bounds in $W^{-1,q}(\T^2)$ with $q$ as above. The stability argument mentioned above translates this into convergence of $(\tphi_\eta)$ in $\cC(\R_+;H^1(\T^2))$. Then interpolation with uniform bounds respectively in $L^p(\T^2)$ and $W^{2,p}(\T^2)$ provides convergence in norms sufficient to take limits in the equations. 
      \end{proof}
      \subsection{Heavy particles}\label{s:heavy}

      Now, consistently with modeling considerations, for reading convenience we analyze the heavy particle case setting $\sigma=1$. Thus we study the system
      \be
      \left\{
      \begin{aligned}
	&\d_t N\,+\,
	\textrm{div}_{x_\perp}(N\,\nabla_{x_\perp}^\perp\tphi)\,=\,0\,,\\
	&-\delta^2\Delta_x\phi\, =\frac{Ne^{-\phi}}{\int_\T e^{-\phi(\cdot,\cdot,y_\mypar)}\dD y_\mypar} - \frac{e^{\phi}}{\int_{\T^3} e^{\phi(\cdot,y)}\dD\,y}
	\,,\quad \tphi(t,x_\perp)\,=-\ln\left(\int_\T e^{-\phi(t,x_\perp,x_\mypar)}\dD\,x_\mypar\right).
      \end{aligned}
      \right.
      \label{GCPB-heavy}
      \ee

      As in the previous subsection we focus on the elliptic equation involved in System~\ref{GCPB-heavy} and while doing so we temporarily omit to mention time dependence and do not always repeat but always assume $N\geq0$, $\int_{\T^2}N=1$.
      
      Hence we study the following mixed Poisson-Boltzmann equation
      \be
      -\delta^2\Delta_x \phi(x_\perp, x_\mypar)\, =\,N(x_\perp)\frac{e^{-\phi(x_\perp, x_\mypar)}}{\int_\T e^{-\phi(x_\perp, y_\mypar)}\dD y_\mypar}-
      \frac{e^{\phi(x_\perp, x_\mypar)}}{\int_{\T^3} e^{\phi(y)}\dD y}
      \label{PBmixed}
      \ee
      where  $N$ is a nonnegative function with integral equal to one. Our first observation is that Equation~\ref{PBmixed} is the Euler-Lagrange equation associated with the energy functional 
      \be
      J[\psi]\, =\, \frac{1}{2}\delta^2\int_{\T^3}\left|\nabla\psi\right|^2 
      - \int_{\T^2}N\, \ln\left(\int_\T e^{-\psi}\, \dD x_\mypar\right)\,\dD x_\perp
      + \ln\left(\int_{\T^3} e^{\psi}\,\right)\,.
      \label{energ-heavy}
      \ee

      Again we consider 
      \[
      H_0\,=\,\left\{\quad h\in\mathcal{D}'(\T^3)\quad|\quad\nabla h \in L^2(\T^3)\ \textrm{and}\ \int_{\T^3}h = 0\quad\right\}.
      \]
      With minor modifications on the proof of Proposition~\ref{propJ}, one obtains the following.
      
      \bpr Assume $N\in L^{4/3}(\T^2)$. Then
      \begin{enumerate}
	\item[(i)] $J:H_0\rightarrow\R\cup\{+\infty\}$ is well defined, bounded by below on bounded sets and coercive;
	\item[(ii)] $J$ is strictly convex;
	\item[(iii)] $J$ is Gâteaux differentiable at $\phi$ in the direction $\psi-\phi$ for any $(\phi,\psi)\in (H_0)^2$ such that $J(\phi)<+\infty$ and $J(\psi)<+\infty$ ;
	\item[(iv)] $J$ is weakly lower semi-continuous.
      \end{enumerate}
      \label{propJ-heavy}
      \epr
      
      \bpr[Existence and uniqueness]
      For any $N\in L^{4/3}(\T^2)$, Equation~\eqref{PBmixed} possesses a unique weak solution $\phi\in H_0$ such that $\int_{\T^3}e^\phi<+\infty$. Moreover there exists $C$ such that for any such $N$ the solution $\phi$ satisfies
      $$
      \|\phi\|_{H^1(\T^3)}^2\,+\,\ln\left(\int_{\T^3}e^\phi\right)\,\leq\,
      C\,\|N\|_{L^{4/3}(\T^2)}^2\,.
      $$
      Besides there exists $C$ such that for any pair $N$, $N'$ of such functions respective solutions $\phi$ and $\phi'$ satisfy
      $$
      \|\phi-\phi'\|_{H^1(\T^3)}\,\leq\,
      C\,\|N-N'\|_{L^{4/3}(\T^2)}
      $$
      \label{existPB-heavy}
      \epr
      
      To tight more closely both propositions note that
      $$
      J^{-1}(\R)\,=\,\left\{\quad h\in\mathcal{D}'(\T^3)\quad|\quad\nabla h \in L^2(\T^3)\,,\ \int_{\T^3}h = 0\quad\textrm{and}\quad\int_{\T^3}e^\phi<+\infty\quad\right\}\,.
      $$
      
      Likewise one may also derive qualitative properties of solutions by paralleling the light particle case. However instead we retain only uniqueness from the foregoing propositions and as already mentioned in the introduction we observe that $\phi$ is actually independent of its third argument and may be obtained by solving
      \be
      -\delta^2\Delta_{x_\perp} \phi(x_\perp)\, =\,N(x_\perp)-
      \frac{e^{\phi(x_\perp)}}{\int_{\T^2} e^{\phi(y)}\dD y}\,.
      \label{PBmixed-bis}
      \ee
      To justify this claim it is sufficient now to prove that \eqref{PBmixed-bis} possesses a solution. 
      
      Equation~\eqref{PBmixed-bis} is a classical Poisson-Boltzmann equation. The equation itself and various generalizations have been studied thoroughly elsewhere and the reader is referred to \cite{Bouchut_VP_small-electron-mass,Dolbeault_stationary-states,CDMS} as entering gates to the extensive relevant literature. Here we merely apply slight variations on the arguments of the light species case to obtain the following propositions. Generally speaking the modified arguments turn to be actually simpler than the original ones. Implicitly they involve the consideration of 
      $$
      J_\perp[\psi]\, =\, \frac{1}{2}\delta^2\int_{\T^2}\left|\nabla\psi\right|^2 \dD x
      - \int_{\T^2}N\, \psi
      + \ln\left(\int_{\T^2} e^{\psi}\right)
      $$
      on
      \[
      H_\perp\,=\,\left\{\quad h\in\mathcal{D}'(\T^2)\quad|\quad\nabla h \in L^2(\T^2)\ \textrm{and}\ \int_{\T^2}h = 0\quad\right\}.
      \]
      
      \bpr[Existence and uniqueness]
      For any\footnote{At this level of regularity in the definition of $J_\perp$ the term involving $N$ should be interpreted as a duality pairing.} $N\in H^{-1}(\T^2)$, Equation~\eqref{PBmixed-bis} possesses a unique weak solution $\phi\in H_\perp$ such that $\int_{\T^2}e^\phi<+\infty$. Moreover there exists $C$ such that for any such $N$ the solution $\phi$ satisfies
      $$
      \|\phi\|_{H^1(\T^2)}^2\,+\,\ln\left(\int_{\T^2}e^\phi\right)\,\leq\,
      C\,\|N\|_{H^{-1}(\T^2)}^2\,.
      $$
      Besides there exists $C$ such that for any pair $N$, $N'$ of such functions respective solutions $\phi$ and $\phi'$ satisfy
      $$
      \|\phi-\phi'\|_{H^1(\T^2)}\,\leq\,
      C\,\|N-N'\|_{H^{-1}(\T^2)}
      $$
      \label{existPB-heavy-bis}
      \epr

      \bpr[Regularity]
      Let $1<p_0\leq\infty$. There exists $C$ such that if $N\in L^{p_0}(\TT^2)$ then $\phi$ the unique solution to \eqref{PBmixed-bis} satisfies
      $$
      \|\phi\|_{L^\infty(\T^2)}\leq 
      C\,\left(1+\|N\|_{L^{p_0}(\TT^2)}\right)\,.
      $$
      Moreover, for any $1<p<\infty$, there exists $C$ such that if $N\in L^{p_0}(\TT^2)\cap L^p(\TT^2)$ then $\phi$ the unique solution to \eqref{PBmixed} satisfies
      $$
      \|\phi\|_{W^{2,p}(\T^2)}\leq 
      C\,\left(e^{C\,\|N\|_{L^{p_0}(\TT^2)}}+\|N\|_{L^p(\TT^2)}\right)\,.
      $$
      Besides for any $1<p<\infty$ and any $1<q<\infty$ such that $\frac1q\leq\frac1p+\frac12$
      there exists $C$ such that with if $N\in L^q(\TT^2)\cap L^{p_0}(\TT^2)\cap W^{1,p}(\TT^2)$ 
      then $\phi$ the unique solution to \eqref{PBmixed} satisfies
      $$
      \|\phi\|_{W^{3,p}(\T^3)}\,\leq\,
      C\,\Big(\|N\|_{W^{1,p}(\TT^2)}
      \,+\,\Big(\|N\|_{L^{q}(\TT^2)}+1\Big)\,e^{C\,\|N\|_{L^{p_0}(\TT^2)}}\Big)\,.
      $$
      \label{regellip-heavy}
      \epr
      \begin{proof}
	On one hand we first observe that $-\delta^2 \Delta\phi\leq N$ so that Lemma~\ref{maxprinc-2D} and the coercivity estimate of Proposition~\ref{existPB-heavy-bis} provides the needed upper bound on $\phi$. On the other hand we also have $-\delta^2 \Delta\phi\geq -e^\phi$. Now note that on any bounded $\Omega'$ one may chose $c_0\in\R$ to ensure that $-e^{\|\cdot\|^2+c_0}\geq -\delta^2\Delta(\|\cdot\|^2+c_0)$. With suitable $\Omega$, $\Omega'$ and $c_0$ one may then apply Lemma~\ref{maxprinc-2D-bis} to $u=\|\cdot\|^2+c_0-\phi$ with $K_0=0$. Combined with the coercivity estimate of Proposition~\ref{existPB-heavy} this yields the claimed upper bound on $\phi$ thus achieves the proof of the estimate on $\|\phi\|_{L^\infty(\R)}$.
	
	From here the proof is achieved along the lines of the proof of the light species case.
      \end{proof}
      
      \bpr[Lipschitz dependence]
      For any $1< p<\infty$ and $1<p_0\leq\infty$, there exists $C$ such that if $N\in L^p(\TT^2)\cap L^{p_0}(\TT^2)$, $N'\in L^p(\TT^2)\cap L^{p_0}(\TT^2)$ then $\phi$ and $\phi'$ the respective solutions to \eqref{PBmixed-bis} corresponding to $N$ and $N'$ satisfy
      $$
      \|\phi-\phi'\|_{W^{2,p}(\T^3)}
      \,\leq\,
      C\,\Big(\|N-N'\|_{L^p(\T^2)}\,+\,
      \|N\|_{H^{-1}(\TT^2)}
      \left(e^{C\,\|N\|_{L^{p_0}(\TT^2)}}+e^{C\,\|N'\|_{L^{p_0}(\TT^2)}}\right)\Big)\,.
      $$
      \label{Lipschitz-dependence-heavy}
      \epr

      From here proofs of Theorems~\ref{th:strong} and \ref{th:weak} in the ionic case follow arguments completely parallel to those of the electronic case that we omit here.
      
      \section{Mathematical justification of the limiting process}\label{s:maths}
      In this section, we investigate the rigorous derivation of the asymptotic model \eqref{guiding-center}-\eqref{averaged-phi} from original equation \eqref{FPf} in a linear setting, that is when electric forces derive from an external electric potential.
      
      To do so, we provide in Subsection~\ref{s:apriori} an existence result and corresponding \textit{a priori} estimates yielding bounds on solutions uniform with respect to $\eps$. In Subsection~\ref{s:compactness} we prove that those bounds provide enough compactness on the $\eps$-family of solutions to justify formal asymptotic arguments. This leads to Theorem~\ref{mainCV}, the first main result of this section. 
      
      In the last subsection, we investigate further the asymptotic process yielding convergence of the macroscopic density to a limiting anisotropic Boltzmann-Gibbs density. Specifically we provide decay rates in terms of $\eps$ and the exponent $\alpha$. To achieve this task, we use a different functional setting that is more demanding on the localization of considered initial data but classical for the Fokker-Planck operator. Our task is then achieved by designing an anisotropic hypocoercive strategy based on a suitable global Lyapunov functional. This results in Theorem~\ref{th:hypo}.  
      
      \subsection{\textit{A priori} estimates}\label{s:apriori}
      
      We have already encountered in the formal derivation section, either implicitly or explicitly, the two main \textit{a priori} estimates at our disposal, conservation of mass
      $$
      \iint_{\T^3\times\R^3} f^\eps(t,x,v)\,\dD x\dD v
      \,=\,\iint_{\T^3\times\R^3} f_0(x,v)\,\dD x\dD v
      $$
      obtained by integrating \eqref{FPf} and free energy dissipation \eqref{freeenerg}  obtained by setting $H(h)=h\ln(h)$ in \eqref{e:H} 
      $$
      \begin{array}{rcl}\ds
	\iint_{\T^3\times\R^3}
	f^\eps(t,x,v)&&\ds\hspace{-1em}
	\ln\left(\frac{f^\eps(t,x,v)}{\cM(t,x,v)}\right)\,\dD x\dD v
	\ +\ 4\,\eps^{-(1+\alpha)}\,\int_0^t\iint_{\T^3\times\R^3} \left|\nabla_v\sqrt{\frac{f^\eps(s,x,v)}{M(v)}}\right|^2\,M(v)\,\dD x\dD v\,\dD s\\
	&\hspace{-2em}=&\ds\hspace{-1em}
	\iint_{\T^3\times\R^3} f_0(x,v)\ln\left(\frac{f_0(x,v)}{\cM(0,x,v)}\right)\,\dD x\dD v
	\,-\,
	\sigma\int_0^t\iint_{\T^3\times\R^3} \d_t\phi(s,x)\,f^\eps(s,x,v)\,\dD x\dD v\,\dD s
      \end{array}
      $$
      where $\cM$ is the global Maxwellian defined in \eqref{BG} that has integral equal to one by our normalization\footnote{Note that this differs from normalization in Section~\ref{s:WP}.} \eqref{phi-normalized} of $\phi$. Note however that, as is customarily expected, our weak solutions will satisfy only a relaxed form of the foregoing free energy dissipation equality.
      
      To unfold controls provided by the above estimates we recall the following well-known lemma.
      
      \bl\label{l:log}
      There exists a constant $C$ such that for any integrable $f:\T^3\times\R^3\to \R_+^*$ and any bounded $\phi:\R_+\times\T^3\to\R$
      $$
      \begin{array}{rcl}\ds
	\hspace{-0.25em}\iint_{\T^3\times\R^3}&&\hspace{-2em}\ds(f(x,v)
	\,|\ln(f(x,v))|\,+\,|v|^2\,f(x,v))\dD x\dD v\\[0.5em]
	\hspace{-1em}&\hspace{-1.5em}\leq&\hspace{-1em}\ds
	C\,\Big((1+\ln_-(\|f\|_{L^1(\T^3\times\R^3)})+\|\phi\|_{L^\infty(\T^3)})\,\|f\|_{L^1(\T^3\times\R^3)}
	\,+\hspace{-0.5em}\iint_{\T^3\times\R^3}
	f(x,v)\ln\left(\frac{f(x,v)}{e^{-\sigma\,\phi(x)}\,M(v)}\right)\,\dD x\dD v\Big)\,.
      \end{array}
      $$
      \el
      
      \begin{proof}
	We start by expanding
	$$
	\iint_{\T^3\times\R^3}
	f(x,v)\ln\left(\frac{f(x,v)}{e^{-\sigma\,\phi(x)}\,M(v)}\right)\,\dD x\dD v
	$$
	into
	$$
	\iint_{\T^3\times\R^3}
	(\,|\ln(f(x,v))|\,+\,|v|^2\,+\,\sigma\,\phi(x)\,+\,\tfrac32\,\ln(2\pi)\,)\,f(x,v)\dD x\dD v
	\,-\,2\,\iint_{\T^3\times\R^3}f(x,v)\,\ln_- f(x,v)\dD x\dD v
	$$
	where $(\cdot)_-$ denotes negative part. Then we note that for any positive $u$, $u_0$, 
	$$
	u\ln_-(u)\ \leq\ \begin{cases}
	  \quad u\ln_-(u_0)\quad&\quad\textrm{if }u\geq u_0\\
	  \quad C\,\sqrt{u_0}\quad&\quad\textrm{if }u\leq u_0
	\end{cases}
	$$
	where $C=\max_{w\in\R_+^*} \sqrt{w}\ln_-(w)$. Applying the above to $u=f(x,v)$, $u_0=C_0e^{-\beta\,|v|^2}$ for each fixed $(x,v)$ yields for any $\beta>0$, $C_0>0$
	$$
	\iint_{\T^3\times\R^3}f(x,v)\,\ln_- f(x,v)\dD x\dD v
	\,\leq\,
	C\sqrt{C_0}\iint_{\T^3\times\R^3}e^{-\beta\,|v|^2/2}\dD x\dD v + \iint_{\T^3\times\R^3}(\beta|v|^2+\ln_-(C_0))f(x,v)\dD x\dD v.
	$$
	We conclude by choosing first some $\beta\in(0,1/2)$ then optimizing in $C_0$.
      \end{proof}
      
      Relying on the free energy dissipation one may then obtain the following existence result.
      
      \bpr
      Assume that $\phi\in W^{1,1}_{loc}(\R_+;L^\infty(\T^3))$ and $f_0$ is such that
      \[
      f_0\geq 0,\quad \iint_{\T^3\times\R^3}\left(1 + |v|^2 + \ln_+ f_0\right)\,f_0\, \dD x \dD v<\infty\,.
      \]
      Then there exists a unique $f^\eps\in \cC(\R_+, L^1(\T^3\times\R^3))$ solving Equation~\eqref{FPf} and starting from $f_0$ at time $0$. Moreover there exists $C$ such that for any such $(\phi,f_0)$ this solution additionally satisfies for any $t\geq0$
      \begin{multline}
	\iint_{\T^3\times\R^3}\left(1 + |v|^2 + \left|\ln f^\eps(t,x,v)\right|\right)\,f^\eps(t,x,v)\, \dD x\, \dD v\ +\ 4\,\eps^{-(1+\alpha)}\,\int_0^t\iint_{\T^3\times\R^3} \left|\nabla_v\sqrt{\frac{f^\eps}{M}}\right|^2\,M\\
	\leq C\,\Big[\left(1+\ln_-(\|f_0\|_{L^1(\T^3\times\R^3)})+ \|\phi\|_{W^{1,1}((0,t);L^\infty(\T^3))}\right)\,\|f_0\|_{L^1(\T^3\times\R^3)}\\
	\,+\,\iint_{\T^3\times\R^3}\left(|v|^2 + \ln_+ f_0(x,v)\right)\,f_0(x,v)\, \dD x \dD v\Big]\,.
	\label{mainapriori}
      \end{multline}
      \label{existFPlin}
      \epr
      
      Bounds on free energy provide enough compactness to reduce the proof of the existence part of the proposition to an existence proof for a dense subset of initial data. Moreover uniqueness may also be obtained by proving an existence result for a dual equation starting from a suitably large set of initial data. Besides one may choose the measure with respect to which identification  of $L^2$-duals is performed in order to make this dual equation as simple as possible. Here one may enforce that the dual equation has essentially the same structure as the original one ; see Section~\ref{s:hypo} for  details on involved symmetries. At last existence for data in adapted Sobolev spaces --- which is sufficient to complete the proof --- follows from classical semigroup arguments \cite{Pazy}. One may for instance use a spectral version of classical hypocoercive estimates to prove that the equation generates a semigroup when $\phi\equiv0$ \cite[Chapter~4]{Pazy} and recover the full case by a fixed point argument or directly establish frozen-time spectral estimates to prove that the equation generates an evolution system \cite[Chapter~5]{Pazy}. We skip these truly classical details but refer the reader to \cite[Chapter~5]{Helffer-Nier_book} for a detailed analysis that is closely related to the last part of the argument.
      
      \subsection{Weak compactness arguments}\label{s:compactness}
      
      As already announced it turns out that in the linear case the arguments that we have expounded along the formal derivation in Section \ref{s:heuristic} have fully justified close counterparts. This is the line we follow to prove Theorem~\ref{th:compactness}.

      \begin{proof}[Proof of Theorem \ref{th:compactness}]
	Since $(f^\eps\,M^{-1}(1+(\ln(f^\eps\,M^{-1}))_+))$ is uniformly bounded in $L^1_{loc}(\R_+;L^1(\T^3\times\R^3,M(v)\,\dD x\,\dD v))$, de la Vall\'ee-Poussin and Dunford-Pettis theorems imply that $(f^\eps)$ converges weakly to some $f$ in $L^1_{loc}(\R_+;L^1(\T^3\times\R^3))$ along some subsequence. Moreover the bound on the dissipation of free energy in \eqref{mainapriori} implies that 
	$$
	\int_0^t\iint_{\T^3\times\R^3} \left|\nabla_v\sqrt{\frac{f(s,x,v)}{M(v)}}\right|^2\,M(v)\,\dD x\dD v\,\dD s\,=\,0
	$$
	thus $f$ is a local Maxwellian, $f=n\,M$ with $n=\int_{\R^3} f(\cdot,\cdot,v)\,\dD v$. Besides tensorization with the  constant function equal to one on $\R^3$ shows that $n^\eps$ does converge weakly to $n$ in $L^1_{loc}(\R_+;L^1(\T^3))$.
	
	Bounds on the dissipation also shows that $(\eps^{-\alpha}(v\,f^\eps+\nabla_v f^\eps))$ converges strongly to zero in $L^2(\R_+;L^1(\T^3\times\R^3))$ since
	$$
	v\,f^\eps+\nabla_v f^\eps\,=\,2\,\sqrt{f^\eps}\quad \nabla_v\left(\sqrt{\frac{f^\eps}{M}}\right)\,\sqrt{M}
	$$
	and the Cauchy-Schwarz inequality shows that the relevant norm is $\cO(\eps^{(1-\alpha)/2})$ as $\eps\to0$. This key observation shows that for any smooth compactly-supported function $\psi:\R_+\times\T^3\times\R^3\to\R$ that is independent of the polar angle $\theta$, or equivalently such that $v^\perp\cdot\nabla_v \psi\equiv 0$, taking the limit $\eps\to0$ along the relevant subsequence yields
	$$
	\int_{\R_+\times\T^3\times\R^3}
	(v\cdot\nabla_x\psi(t,x,v)\,-\,\sigma\nabla_x\phi(t,x)\cdot\nabla_v\psi(t,x,v))\,f(t,x,v)\,\dD t\,\dD x\,\dD v\,=\,0
	$$
	which, given that $f$ is a local Maxwellian, is also 
	$$
	\int_{\R_+\times\T^3\times\R^3}
	v\cdot(\nabla_x\psi(t,x,v)\,-\,\sigma\nabla_x\phi(t,x)\ \psi(t,x,v))\,n(t,x)\,M(v)\,\dD t\,\dD x\,\dD v\,=\,0
	$$
	thus, by symmetries of $\psi$ and $M$,
	$$
	\int_{\R_+\times\T^3\times\R^3}
	v_\mypar(\d_{x_\mypar}\psi(t,x,v)\,-\,\sigma\d_{x_\mypar}\phi(t,x)\ \psi(t,x,v))\,n(t,x)\,M(v)\,\dD t\,\dD x\,\dD v\,=\,0\,.
	$$
	Choosing a non zero smooth radial compactly-supported $\tpsi:\R^3\to\R_+$ and noticing that the foregoing class of $\psi$ functions includes $(t,x,v)\mapsto v_{\mypar}\,\tpsi(v)\,\upsi(t,x)$ for any smooth compactly-supported $\upsi:\R_+\times\T^3\to\R_+$, one derives that in distributional sense $\d_{x_\mypar}n\,=\,-\sigma\,n\,\d_{x_\mypar}\phi$. Hence $n\in L^1_{loc}(\R_+\times\T^2;W^{1,1}(\T))$ and $\d_{x_{\mypar}}(e^{\sigma\,\phi}\,n)=0$ thus
	$$
	n(t,x)\,=\,N(t,x_\perp)\,\frac{e^{-\sigma\phi(t,x)}}{\int_{\T} e^{-\sigma\phi(t,x_\perp,y_\mypar)}\,\dD\,y_\mypar}\,,\qquad N(t,x_\perp)\,=\,\int_{\T}n(t,x_\perp,y_\mypar)\dD y_\mypar\,.
	$$
	
	Now we prove that $N$ satisfies the claimed equation. To do so we set 
	$$
	N^\eps(t,x_\perp)\,=\,\int_{\T}n^\eps(t,x_\perp,y_\mypar)\,\dD y_\mypar
	$$ 
	and observe that by using again a tensorization argument one gets that along the relevant subsequence $N^\eps$ converges weakly to $N$ in $L^1_{loc}(\R_+;L^1(\T^2))$. As in Section~\ref{s:heuristic} we also introduce
	$$
	J^\eps(t,x_\perp)\,=\,\frac{1}{\eps}\int_{\T\times\R^3}\,v\,f^\eps(t,x_\perp,y_\mypar,v)\,\dD y_\mypar\,\dD v
	$$
	and observe that $\d_tN^\eps\,+\,\textrm{div}_{x_\perp}(J^\eps)\,=\,0$. The same kind of integration shows that
	$$
	\sigma(J^\eps)^\perp+
	\textrm{div}_{x_\perp}\left(\int_{\T\times\R^3}\,v_\perp\otimes v_\perp\,f^\eps(\cdot,\cdot,x_\mypar,v)\dD x_\mypar \dD v\right)
	\ +\ \sigma\int_{\T}\nabla_{x_\perp}\phi(\cdot,\cdot,x_\mypar)\,n^\eps(\cdot,\cdot,x_\mypar)\dD x_\mypar
	$$ 
	converges weakly to zero. Moreover the last term of the foregoing expression converges weakly to 
	$$
	\sigma\int_{\T}\nabla_{x_\perp}\phi(\cdot,\cdot,x_\mypar)\,n(\cdot,\cdot,x_\mypar)\dD x_\mypar
	\ =\ -N\,\frac{\nabla_{x_\perp}\left(\int_{\T}e^{-\sigma\,\phi(\cdot,\cdot,x_\mypar)}\dD x_\mypar\right)}{\int_{\T}e^{-\sigma\,\phi(\cdot,\cdot,x_\mypar)}\dD x_\mypar}
	$$
	along the relevant subsequence. Finally we observe that
	$$
	\begin{array}{rcl}\ds
	  \int_{\T\times\R^3}\,v_\perp\otimes v_\perp
	  \,f^\eps(\cdot,\cdot,x_\mypar,v)\dD x_\mypar \dD v
	  &=&\ds
	  \int_{\T\times\R^3}\,v_\perp\otimes(v_\perp+\nabla_{v_\perp})
	  \,f^\eps(\cdot,\cdot,x_\mypar,v)\dD x_\mypar \dD v
	  \,+\,N^\eps\,\I
	\end{array}
	$$
	and that 
	$$
	\begin{array}{rcl}\ds
	  \int_{\T\times\R^3}&&\ds
	  v_\perp\otimes(v_\perp+\nabla_{v_\perp})
	  \,f^\eps(\cdot,\cdot,x_\mypar,v)\dD x_\mypar \dD v\\
	  &\hspace{-2em}=&\ds
	  2\int_{\T\times\R^3}\,v_\perp\,\sqrt{f^\eps(\cdot,\cdot,x_\mypar,v)}\otimes\nabla_{v_\perp}
	  \,\left(\sqrt{\frac{f^\eps}{M}}\right)(\cdot,\cdot,x_\mypar,v) \sqrt{M(v)}\dD x_\mypar \dD v
	\end{array}
	$$
	is seen to converge to zero in $L^2(\R_+;L^1(\T^2))$ by using the Cauchy-Schwarz inequality combined with bounds on second-order moments and dissipation of free energy. Therefore $J^\eps_\perp$ converges weakly to 
	$$
	\sigma(\nabla_{x_\perp}N)^\perp\,-\,\sigma\,N\,\left(\nabla_{x_\perp}\ln\left(\int_{\T}e^{-\sigma\,\phi(\cdot,\cdot,x_\mypar)}\dD x_\mypar\right)\right)^\perp
	$$
	along the relevant subsequence. Inserting this in $\d_tN^\eps\,+\,\textrm{div}_{x_\perp}(J^\eps)\,=\,0$ achieves the proof.
      \end{proof}
      \br
      It may be instructive to compare the former proof with the arguments proving \cite[Lemma~4.1]{herda_2015_massless} as those would apply equally well here. They rely rather on a direct compactness argument on $n^\eps$ and log-Sobolev and Czisz\'ar-Kullback-Pinsker inequalities to control $f^\eps-n^\eps\,M$.
      \er
      \br
      Note that in the last step while we had not enough control on moments in velocities to justify in this way the convergence of expressions involving terms quadratic in velocity we were able to use cancellations and collisional dissipation to do so. This is a by now classical trick that was crucial in similar studies of the diffusive regime $\alpha=1$ ; see \cite[Section~5]{poupaud_2000_parabolic} and \cite[Propositions~3.3 \&~4.3]{herda_2015_massless}. However we could replace it with an argument involving evanescent cut-off functions as in \cite{StRaymond_large-velocities}. 
      \er
      
      \subsection{Anisotropic hypocoercive estimates}\label{s:hypo}
      Now we provide convergence rates for some limits involved in the derivation of our asymptotic model. To keep technicalities as low as possible we assume here that $\phi$ does not depend on time $t$ but only on $x\in\T^3$. While it is likely that one may also deal with time-dependent cases assuming sufficient control on time variations of $\phi$ this would add unessential complications to already technical estimates.
      
      Part of the asymptotic process involves oscillatory processes and some limits are in essence weak limits that may only be expected to hold in spaces of negative regularity. We only focus here on those that are expected, and proved here, to be strong limits. With this in mind the goal of this section is merely to provide quantitative bounds, explicit with respect to $\eps$ and $\alpha$, on the distance between the distribution function $f^\eps$ and an anisotropic Maxwell-Boltzmann density expressed in terms of $N^\eps$. While one could track convergence rates of vanishing terms in the foregoing convergence proof, this would lead to estimates in functional spaces of (highly) negative regularity by lack of control on (spatial) gradients of the solution. Here instead we prove that by assuming that initial data belongs to a stronger space, controlling in particular any number of moments in velocity, we obtain natural convergence rates for 
      $$
      f^\eps(t,x,v)\,-\,N^\eps(t,x_\perp)\,\frac{e^{-\sigma\phi(x)}}{\int_{\T} e^{-\sigma\phi(x_\perp,y_\mypar)}\,\dD\,y_\mypar}\,\frac{e^{-\frac{|v|^2}{2}}}{(2\pi)^{3/2}}
      $$
      in strong norms for some suitable $N^\eps$. Note that while convergence of $f^\eps-n^\eps\,M$ follows at once from the dissipation of free energy, the next step relies on a subtler interplay between three kind of terms : electric-and-free transport, magnetic transport and collisions. As we show in Appendix~\ref{s:hypocoercapp} this interplay may be essentially elucidated by computing commutators of involved terms and then encoded in a global functional of hypocoercive type.
      
      To make the functional setting more precise, let us introduce $h^\varepsilon  = f^\varepsilon/\cM$ where once again
      \[
      \cM(x,v)\, =\,\frac{1}{(2\pi)^{3/2}} e^{-\left(\sigma\phi(x)+\frac{|v|^2}{2}\right)},
      \] 
      is our global Maxwellian. Associated initial data $h_0  = f_0/\cM$ is assumed to belong to $\cH=L^2(\mu)$ the Hilbert space $\cH$ characterized by its
      $L^2$-norm $\|\cdot\|=(\int |\cdot|^2\,\dD\mu)^{1/2}$ where $\mu$ is the probability measure with density $\cM$. In the following we denote by $\langle\cdot;\cdot\rangle$ the corresponding scalar product.
      
      A direct computation shows that the distribution function $f^\varepsilon$ solves the linear Vlasov-Fokker-Planck equation \eqref{FPf} with external electric potential $\phi$ if and only if $h^\varepsilon$ satisfies
      \begin{equation}
	\varepsilon\ \partial_t h^\eps + v\cdot\nabla_xh^\eps - \sigma\nabla_x\phi\cdot\nabla_vh^\eps  - \frac{\sigma}{\varepsilon}v^\perp\cdot\nabla_vh^\eps\, = \,\frac{1}{\varepsilon^\alpha}\left(v\cdot\nabla_vh^\eps - \Delta_vh^\eps\right).
	\label{FPh}
      \end{equation}
      In the absence of any magnetic field, the original version of the approach that we extend here to serve our purposes would lead to a quantitative proof of convergence as $\eps\to0$ to the global Maxwellian $\cM$, and in this case it would really follow from very classical hypocoercive estimates. See \cite{Villani_hypo,Mouhot-Neumann,Dolbeault-Mouhot-Schmeiser,HR_hypo} for some instances of proofs of hypocoercive convergence to equilibrium through global functionals. The presence of a strong magnetic field prevents return to equilibrium in the perpendicular direction and our goal is to design functionals that capture this anisotropic behavior. We detail this strategy along the proof of the following theorem, which is the main result of this section.
      
      \begin{theorem}\label{th:hypofull}
	Suppose that $\phi\in W^{2,\infty}(\T^3)$.\\
	Then for any $h_0\in\cH$ and any $\eps>0$ there exists a unique global solution $h^\eps$ to \eqref{FPh} starting from $h_0$ at time zero such that $h^\eps\in\cC(\R_+;\cH)$ and $\nabla_vh^\eps\in L^2(\R_+;\cH)$.\\
	Moreover there exists a constant $C_\phi>0$, depending only on the norm of $\phi$ in $W^{2,\infty}(\T^3)$, such that for any $\varepsilon\in(0,1)$, $f^\eps=h^\eps\cM$ --- solving \eqref{FPf} with initial data $f_0=h_0\cM$ and associated to the above solution $h^\eps$ --- satisfies
	\be
	\left\|f^\varepsilon-n^\varepsilon M\right\|_{L^2(\R_+;\ L^2(M^{-1}(v)\dD x\dD v))}\ \leq\ C_\phi\,\|h_0\|\,\varepsilon^{\frac{\alpha+1}{2}};
	\label{distMaxw}
	\ee
	\be
	\left\|n^\varepsilon-N^\eps\,\frac{e^{-\sigma\phi}}{\int_{\T} e^{-\sigma\phi(\cdot,y_\mypar)}\,\dD\,y_\mypar}\right\|_{L^2(\R_+\times\T^3)}\ \leq\ C_\phi\,\|h_0\|\,\varepsilon^{\frac{1-|\alpha|}{2}}
	\label{distBG}
	\ee
	where
	$$
	n^\eps\,=\,\int_{\R^3} f^\eps(\cdot,\cdot,v)\dD v
	\qquad\textrm{and}\qquad
	N^\eps\,=\,\int_{\T}n^\eps(\cdot,\cdot,y_\mypar)\dD y_\mypar\,.
	$$
	In particular, since $\alpha\in(-1,1)$, as $\eps\rightarrow 0$ the distribution function $f^\eps$ gets $\eps^{\frac{1-|\alpha|}{2}}$-close to the anisotropic Maxwell-Boltzmann density 
	$$
	(t,x,v)\mapsto N^\eps(t,x_\perp)\,\frac{e^{-\sigma\phi(x)}}{\int_{\T} e^{-\sigma\phi(x_\perp,y_\mypar)}\,\dD\,y_\mypar}\,\frac{e^{-\frac{|v|^2}{2}}}{(2\pi)^{3/2}} 
	$$
	in $L^2(\R_+;\ L^2(M^{-1}(v)\dD x\dD v))$.
      \end{theorem}
      \br
      Again we stress however that the last part of the asymptotic description --- the identification of an asymptotic dynamics for limits of $N^\eps$ --- proceeds from averaging mechanisms and we do not expect to be able to capture it by similar dissipative  arguments nor do we expect the corresponding convergence to hold in strong norms.
      \er
      \begin{proof}
	We skip the proof of the existence and uniqueness part as fairly classical but we refer the reader to \cite[Chapter~5]{Helffer-Nier_book} for a detailed analysis of a nearly identical problem. The existence proof also provides detailed justification for the rather formal computations that we perform below. The reader may consult \cite{dMRV} for an instance of a similar detailed verification starting directly from the approximation process proving existence.
	
	The starting point is the $L^2$-estimate obtained from \eqref{e:H} by choosing $H(h)=h^2$, 
	\[
	\|h^\varepsilon(t,\cdot,\cdot)\|^2 + 2\varepsilon^{-(\alpha+1)}\,\int_0^t\|\nabla_{v}h^\eps(s,\cdot,\cdot)\|^2\,\dD s\ \leq\ \|h^\varepsilon_0\|^2.
	\]
	From this, \eqref{distMaxw} follows by Poincaré inequality for the Gaussian measure and boundedness of $\phi$, since 
	\[
	\begin{aligned}
	  &\int_{\R_+\times\TT^3\times\R^3}|f^\varepsilon-n^\varepsilon M|^2 M^{-1}&\leq&&&
	  \int_{\R_+\times\TT^3\times\R^3}\left|\frac{f^\varepsilon}{M}-n^\varepsilon\right|^2 M\\
	  &&\leq&&&\int_{\R_+\times\TT^3\times\R^3}\left|\nabla_v\left(\frac{f^\varepsilon}{M}\right)\right|^2 M\\
	  &&\leq&&&C\, \int_0^{+\infty}\|\nabla_{v}h^\eps(t,\cdot,\cdot)\|^2\,\dD t
	\end{aligned}
	\]
	for some $C$ depending only on $\|\phi\|_{L^\infty(\T^3)}$. In order to prove \eqref{distBG}, we need to extend this strategy that only yields dissipation of velocity derivatives so that it also provides estimates on an $x_\mypar$-derivative. To do so we exploit the hypocoercive structure and modify the $L^2$ functional in a $H^1$ functional whose dissipation contains new terms yielding suitable control on $\d_{x_\mypar}(f^\eps\,\cM^{-1})$. To state this in a concise way we define the $\eps$ and time dependent norm
	$$
	\begin{array}{rcl}
	  \norm{\eps}{t}{h}^2&=&\ds
	  \|h\|^2
	  \,+\,\gampar{1}\,\eps^{(|\alpha|-\alpha)/2}\,\min\left(\{1, \tfrac{t}{\varepsilon^{1+\alpha}}\}\right)\,\|\Apar h\|^2\\[1em]
	  &+&\ds
	  \gampar{2}\,\eps^{|\alpha|+(|\alpha|+\alpha)/2}\,\min\left(\{1, \tfrac{t}{\varepsilon^{1+\alpha}}\}\right)^3\,\|\Cpar h\|^2
	  \,+\,2\,\gampar{3}\,\eps^{|\alpha|}\,\min\left(\{1, \tfrac{t}{\varepsilon^{1+\alpha}}\}\right)^2\,\langle\Apar h,\Cpar h\rangle\\[1em]
	  &+&\ds
	  \gamperp{1}\,\eps^{1-\alpha}\,\min\left(\{1, \tfrac{t}{\varepsilon^{1+\alpha}}\}\right)\,\|\Aperp h\|^2
	  \,+\,\gamperp{2}\,\eps^{2}\,\min\left(\{1, \tfrac{t}{\varepsilon^{1+\alpha}}\}\right)^3\,\|\Cperp h\|^2\\[1em]
	  &+&2\,\gamperp{3}\,\eps^{2-\alpha}\,\min\left(\{1, \tfrac{t}{\varepsilon^{1+\alpha}}\}\right)^2\,\langle\Aperp h,\Cperp h\rangle
	\end{array}
	$$
	and the corresponding $\eps$ and time dependent (partial) dissipation
	$$
	\begin{array}{rcl}
	  \dnorm{\eps}{t}{h}&=&\ds
	  \eps^{-(1+\alpha)}\|\nabla_vh\|^2
	  \,+\,\eps^{-1+(|\alpha|-\alpha)/2-\alpha}\,\min\left(\{1, \tfrac{t}{\varepsilon^{1+\alpha}}\}\right)\,\|\nabla_v\Apar h\|^2\\[1em]
	  &+&\ds
	  \,\eps^{-1+|\alpha|+(|\alpha|-\alpha)/2}\,\min\left(\{1, \tfrac{t}{\varepsilon^{1+\alpha}}\}\right)^3\,\|\nabla_v\Cpar h\|^2
	  \,+\,\eps^{|\alpha|-1}\,\min\left(\{1, \tfrac{t}{\varepsilon^{1+\alpha}}\}\right)^2\,\|\Cpar h\|^2\\[1em]
	  &+&\ds
	  \,\eps^{-2\alpha}\,\min\left(\{1, \tfrac{t}{\varepsilon^{1+\alpha}}\}\right)\,\|\nabla_v\Aperp h\|^2
	  \,+\,\eps^{1-\alpha}\,\min\left(\{1, \tfrac{t}{\varepsilon^{1+\alpha}}\}\right)^3\,\|\nabla_v\Cperp h\|^2\\[1em]
	  &+&\,\eps^{1-\alpha}\,\min\left(\{1, \tfrac{t}{\varepsilon^{1+\alpha}}\}\right)^2\,\|\Cperp h\|^2\,.
	\end{array}
	$$
	These quantities are related through the dissipation estimate contained in the following lemma.
	\bl\label{l:technical}
	Let $\phi\in W^{2,\infty}(\T^3)$. There exist positive constants $\gamperp{i}$, $\gampar{i}$, $i\in\{1,2,3\}$ such that, 
	\begin{itemize}
	  \item uniformly in $(t,\eps)\in\R_+\times(0,1)$, the norm $\norm{\eps}{t}{\cdot}$ is equivalent to the square root of
	  $$
	  \begin{array}{rcl}
	    \| \cdot\|^2
	    &+&\ds
	    \eps^{(|\alpha|-\alpha)/2}\,\min\left(\{1, \tfrac{t}{\varepsilon^{1+\alpha}}\}\right)\,\|\Apar (\cdot)\|^2
	    \,+\,\eps^{|\alpha|+(|\alpha|+\alpha)/2}\,\min\left(\{1, \tfrac{t}{\varepsilon^{1+\alpha}}\}\right)^3\,\|\Cpar(\cdot)\|^2\\[1em]
	    &+&\ds
	    \eps^{1-\alpha}\,\min\left(\{1, \tfrac{t}{\varepsilon^{1+\alpha}}\}\right)\,\|\Aperp(\cdot)\|^2
	    \,+\,\eps^{2}\,\min\left(\{1, \tfrac{t}{\varepsilon^{1+\alpha}}\}\right)^3\,\|\Cperp(\cdot)\|^2\,.
	  \end{array}
	  $$
	  \item there exists $K>0$ such that for any $h_0\in\cH$ the solution $h^\eps$ of~\eqref{FPh} starting from $h_0$ satisfies for any $\eps\in(0,1)$ and any $t\geq0$,
	  \begin{equation}
	    \norm{\eps}{t}{h^\eps(t,\cdot,\cdot)}^2\ + K\int_0^t\ \dnorm{\eps}{s}{h^\eps(s,\cdot,\cdot)}\ \dD s\leq\ \| h_0\|^2.
	    \label{estiml2dissipani}
	  \end{equation}
	\end{itemize}
	\el
	The proof of this result is given in  Appendix~\ref{s:hypocoercapp}. There in particular we explain in details how different choices of power of $t$ and $\eps$ arise as necessary constraints to satisfy expected inequalities and how constants $\gamma_{\mypar,i}$ and $\gamma_{\perp,i}$ may then be tuned to close the argument. Nonetheless, we do provide here some partial insights on the choice of the functional. Standard $L^2$-based hypocoercive strategy for the kinetic Fokker-Planck operator suggests to estimate $H^1$ type functionals with mixed derivative to unravel dissipation of space derivatives that was missing in the original $L^2$ computation. Here, direct computations performed in Lemma~\ref{lemhypo} show that
	\begin{eqnarray}
	  \ds\frac{\dD}{\dD t}\lla \Apar h^\eps, \Cpar h^\eps\rra
	  &+&\eps^{-1}\|\Cpar h^\eps\|^2
	  \,=\,\ds\sigma\,\varepsilon^{-1}\lla  (\nabla_vh^\varepsilon\cdot\nabla_x)\d_{x_\mypar}\phi ,\Apar h^\varepsilon\rra\label{estvxpar}\\
	  &-&\ds
	  \eps^{-(1+\alpha)}\left(\lla \Apar h^\eps, \Cpar h^\eps\rra + 2\lla \nabla_v\Apar h^\eps, \nabla_v\Cpar  h^\eps\rra\right)
	  \nonumber\\
	  \ds\frac{\dD}{\dD t}\lla \Aperp h^\varepsilon, \Cperp h^\varepsilon\rra&+&\varepsilon^{-1}\|\Cperp h^\varepsilon\|^2\label{estvxperp}\\
	  &=&\ds-\varepsilon^{-(1+\alpha)}\big{\{}\lla \Aperp h^\varepsilon, \Cperp h^\varepsilon\rra
	  \,+\,2\lla \nabla_v\Aperp h^\varepsilon, \nabla_v\Cperp  h^\varepsilon\rra\big{\}}\nonumber\\
	  &+&\ds\sigma\,\varepsilon^{-1}\lla (\nabla_vh^\varepsilon\cdot\nabla_x)\nabla_{x_\perp}\phi,\Aperp h^\varepsilon\rra
	  \,-\,\sigma\, \varepsilon^{-2}\lla   \Aperp^\perp h^\varepsilon,\Cperp h^\varepsilon\rra.\nonumber
	\end{eqnarray}
	where good useful dissipation terms have been kept on left-hand sides and useless signless terms have been left on the right. When putting everything in one global functional the goal is to balance each bad term with some fraction of good dissipative terms. We stress that in \eqref{estvxperp} the worst term to balance --- the one of order $\eps^{-2}$ --- is precisely the one due to the presence of a strong magnetic field. It stems from non-commutation of $\nabla_v$ and $v^\perp\cdot\nabla_v$ and forces the choice of anisotropic weights in our functional. Finally, we note that the presence of time weights in our functionals only reflects our choice to use hypoelliptic regularizing effects to start with $L^2$ and not $H^1$ initial data but gain instantaneous $H^1$ control on the solution. In particular our choice of the powers of time involved in the above definitions follows directly from the hypoelliptic structure of the operator at hand. For the kinetic Fokker-Planck the method essentially originates in work of H\'erau \cite{Herau_FP-confining} ; see also \cite[Appendix~A.21]{Villani_hypo} or \cite{dMRV}.
	
	We now use dissipative effects afforded by Lemma~\ref{l:technical} to achieve the proof. The dissipation of the energy functional in \eqref{estiml2dissipani} provides a bound on the size
	of the perpendicular and parallel space derivatives of $h^\eps$ in $\cH$.
	Indeed, by non-negativity of $\|\cdot\|_{\eps,t}$ we get from the dissipation in \eqref{estiml2dissipani} that
	\[
	\begin{aligned}
	  &\|\d_{x_\mypar}h^\eps\|_{L^2((\eps^{1+\alpha}, +\infty);\ \cH)}   &\leq&&&  C\,\|h_0\|\,\varepsilon^{\frac{1-|\alpha|}{2}},\\
	  &\|\d_{x_\perp}h^\eps\|_{L^2((\eps^{1+\alpha},+\infty);\ \cH)}  &\leq&&&  C\,\|h_0\|\,\eps^{\frac{\alpha - 1}{2}}
	\end{aligned}
	\]
	for some constant $C$ depending only on $\|\phi\|_{W^{2,\infty}(\T^3)}$. To prove our claim only the first estimate --- on the parallel derivative --- is needed. Nevertheless, let us stress that the second estimate is essentially useless and in any case does not contradict the non-convergence to a global Maxwellian in the perpendicular direction since $\alpha<1$. To proceed we first observe that
	$$
	e^{\sigma\phi}n^\eps\,=\,\int_{\R^3}h^\eps(\cdot,\cdot,v)\,M(v)\,\dD v
	$$
	hence by the Jensen inequality, for any $t\geq0$
	$$
	\|\d_{x_\mypar}(e^{\sigma\phi}n^\eps)(t,\cdot)\|_{L^2(\T^3)}
	\leq \left(\int_{\R^3}\|\d_{x_\mypar}h^\eps(t,\cdot,v)\|_{L^2(\T^3)}^2\,M(v)\,\dD v\right)^{1/2}
	\leq \|e^{\sigma\phi}\|_{L^\infty}^{\tfrac12}\,\|\d_{x_\mypar}h^\eps(t,\cdot,\cdot)\|\,.
	$$
	Therefore introducing the intermediate quantity
	$$
	\tN^\eps\,=\,\int_{\T}n^\eps(\cdot,\cdot,y_\mypar)\,e^{\sigma\phi(\cdot,y_\mypar)}\dD y_\mypar
	\,\times\,\int_{\T}e^{-\sigma\phi(\cdot,y_\mypar)}\dD y_\mypar
	$$
	the classical Poincar\'e inequality on $\T$ yields
	$$
	\begin{array}{rcl}
	  \ds
	  \left\|n^\varepsilon- \frac{\tN^\varepsilon e^{-\sigma\phi}}{\int_{\T} e^{-\sigma\phi(\cdot,y_\mypar)}\dD y_\mypar}\right\|_{L^2((\eps^{1+\alpha},+\infty)\times\T^3)}
	  &\leq&\ds
	  \|e^{-\sigma\phi}\|_{L^\infty}
	  \left\|n^\eps e^{\sigma\phi}- \frac{\tN^\varepsilon}{\int_{\T} e^{-\sigma\phi(\cdot,y_\mypar)}\dD y_\mypar}\right\|_{L^2((\eps^{1+\alpha},+\infty)\times\T^3)}\\[0.5em]
	  &\leq&\|e^{-\sigma\phi}\|_{L^\infty}
	  \|\d_{x_\mypar}(e^{\sigma\phi}n^\eps)\|_{L^2((\eps^{1+\alpha},+\infty)\times\T^3)}\\[0.5em]
	  &\leq&C\|e^{-\sigma\phi}\|_{L^\infty(\T^3)}\|e^{\sigma\phi}\|_{L^\infty(\T^3)}^{\tfrac12}
	  \|h_0\|\,\varepsilon^{\frac{1-|\alpha|}{2}}\,.
	\end{array}
	$$
	This achieves the proof with $\tN^\eps$ instead of $N^\eps$ when combined with
	$$
	\begin{array}{rcl}
	  \ds
	  \left\|n^\varepsilon- \frac{\tN^\varepsilon e^{-\sigma\phi}}{\int_{\T} e^{-\sigma\phi(\cdot,y_\mypar)}\dD y_\mypar}\right\|_{L^2((0,\eps^{1+\alpha})\times\T^3)}
	  &\leq&\ds
	  \eps^{\frac{1+\alpha}{2}}\left\|n^\varepsilon- \frac{\tN^\varepsilon e^{-\sigma\phi}}{\int_{\T} e^{-\sigma\phi(\cdot,y_\mypar)}\dD y_\mypar}\right\|_{L^\infty((0,\eps^{1+\alpha});L^2(\T^3))}\\[0.5em]
	  &\leq&C'\eps^{\frac{1+\alpha}{2}}\|h_0\|
	  \,\leq\,C'\eps^{\frac{1-|\alpha|}{2}}\|h_0\|
	\end{array}
	$$
	for some $C'$ depending only on $\|\phi\|_{W^{2,\infty}(\T^3)}$, where we have used
	$$
	\|\tN^\eps\|_{L^\infty(\R_+;L^2(\T^2))}
	\leq \|e^{-\sigma\phi}\|_{L^\infty(\T^3)}\|e^{\sigma\phi}\|_{L^\infty(\T^3)}
	\|n^\eps\|_{L^\infty(\R_+;L^2(\T^3))}
	$$
	and
	$$
	\|n^\eps\|_{L^\infty(\R_+;L^2(\T^3))}
	\leq \|e^{\sigma\phi}\|_{L^\infty(\T^3)}^{\frac12}
	\|h^\eps\|_{L^\infty(\R_+;\cH)}
	\leq \|e^{\sigma\phi}\|_{L^\infty(\T^3)}^{\frac12}
	\|h_0\|\,.
	$$
	To conclude the proof we simply observe that 
	$$
	N^\eps-\tN^\eps\,=\,\int_{\T} \left(n^\varepsilon(\cdot,\cdot,z_\mypar)- \frac{\tN^\varepsilon \,e^{-\sigma\phi(\cdot,z_\mypar)}}{\int_{\T} e^{-\sigma\phi(\cdot,y_\mypar)}\dD y_\mypar}\right)\dD z_\mypar
	$$
	hence
	$$
	\|N^\eps-\tN^\eps\|_{L^2(\R_+\times\T^2)}
	\leq \left\|n^\varepsilon- \frac{\tN^\varepsilon e^{-\sigma\phi}}{\int_{\T} e^{-\sigma\phi(\cdot,y_\mypar)}\dD y_\mypar}\right\|_{L^2(\R_+\times\T^3)}\,.
	$$
      \end{proof}
      
      \section{Extensions and perspectives}\label{sec:conclu}

We conclude the present contribution with some further technical comments.
      
      \subsubsection*{Non-constant magnetic fields}
      Though for expository purposes we have restricted our presentation to constant magnetic fields there is one class of non-uniform magnetic fields to which our analysis applies with little modifications : magnetic fields with constant direction but varying amplitude $b$. Explicitly we may consider the kinetic equation
      \[
      \eps\ \d_t f^\eps + v\cdot\nabla_xf^\eps+\sigma\,E^\eps\cdot\nabla_vf^\eps - \frac{1}{\varepsilon}\sigma\,b\,v^\perp\cdot\nabla_vf^\eps\, =\, \frac{1}{\eps^\alpha}
      \textrm{div}_v(vf^\eps + \nabla_vf^\eps),
      \]
      where\footnote{The fact that $b$ does not depend on the parallel coordinate $x_\mypar$ is necessary to ensure a divergence-free magnetic field consistently with Gauss's law.} $b:\R_+\times\T^2\to\R_+$, $(t,x_\perp)\mapsto b(t,x_\perp)$ is smooth and bounded away from zero.
      Theorems~\ref{th:strong}, \ref{th:weak}, \ref{th:compactness} and~\ref{th:hypo} still hold in this case provided \eqref{guiding-center} is replaced with
      $$
      \d_t N\,+\,\textrm{div}_{x_\perp}(N\,U_\perp)\,=\,0\qquad
      \text{where}\qquad
      U_\perp \,=\,\frac{1}{b}(\nabla_{x_\perp}\tphi)^\perp-\frac{\sigma}{b^2}(\nabla_{x_\perp} b)^\perp\,.
      $$
      However for larger classes of magnetic fields the question remains fully open and we expect geometric features of the magnetic field under consideration to play a prominent role even at the formal level. With this respect the reader may benefit from a look at the intricate effects already appearing in linear collisionless models where electric fields are neglected \cite{Cheverry_whistler,Cheverry_axisymmetric} and compare them with \cite{FR2,Bostan_2D-VP,FR18}.

      \subsubsection*{Other collision operators}
      From an applicative point of view the main restriction to relax is probably the form of our collision operators $Q(f) = \mathrm{div}_v(v f +\nabla_vf)$. On this topic we first observe that actually Theorem~\ref{th:compactness} does not rely heavily on detailed collision statistics and that this part of our analysis extend to other kinetic collisional operators with one-dimensional kernels (when considered as operators on functions of $v$ only) such as the relaxation / linear BGK operator $Q(f) = n\,M - f$. We do use strong controls on velocities afforded by the Fokker-Planck operator in our proof of Theorem~\ref{th:compactness} but this may be replaced with more robust arguments involving evanescent cut-off functions as in \cite{StRaymond_large-velocities}. In contrast extending our hypocoercive estimates would require a case-by-case analysis. Further work would be needed to extend any part of our analysis to more realistic operators taking into account classical collisional invariants hence exhibiting higher-dimensional kernels. In another direction we stress that implicitly we have only mentioned cases where the presence of a strong magnetic fields affects collisions statistics in a scalar hence isotropic way. Yet one should expect strong magnetic fields to impact scattering mechanisms hence collision frequencies in an anisotropic way. Taking this into account might also introduce anisotropic effects in the Maxwellian distribution in velocity. To the knowledge of the authors, at least at the analytical level, this line of investigation is however widely open. 
      
      \subsubsection*{Validation of nonlinear reduced systems} Leaving aside modeling considerations the main analytical open question is the validation of nonlinear asymptotic models. In \cite{herda_2015_massless} nonlinear asymptotic reductions have been treated successfully for the critical case $\alpha=1$, corresponding to a diffusive limit, but arguments there did not need to benefit from strongly averaging magnetic effects and instead followed the strategy of the magnetic-field-free case treated in \cite{el_2010_diffusion} by proving that in the critical case magnetic contributions could also be suitably bounded by the free energy dissipation. This in particular provides results that do not require any geometric property for the magnetic fields. In contrast we expect that a refined analysis of cancellations would be needed to carry out a nonlinear analysis in the case under consideration here. Moreover we stress that even in the diffusive cases treated in \cite{el_2010_diffusion,herda_2015_massless} for free energy solutions similar to those of Theorem~\ref{th:compactness} uniform \emph{a priori} bounds are too weak to apply compactness arguments at the level of weak solutions and the analysis involves suitably tailored renormalized solutions. On the other side, the nonlinear analysis of the recent \cite{HR_hypo} proving strong uniform bounds by hypocoercive arguments in the non magnetized case seems to be deeply altered by the introduction of a dominant magnetic field. On a yet another distinct direction one may want to test numerically nonlinear models against direct simulations but since phenomena are genuinely three-dimensional and long this would involve long-time computations in six dimensions, 
      a daunting task.

      \appendix
      
      \section{Hypocoercivity : a technical lemma, Lemma~\ref{l:technical}}
      \label{s:hypocoercapp}
      To motivate the actual choice of powers of $\eps$ we prove a slighted extended version of Lemma~\ref{l:technical} for norms involving free powers of $\eps$,
      $$
      \begin{array}{rcl}
	\norm{\eps}{t}{h}^2&=&\ds
	\|h\|^2
	\,+\,\gampar{1}\,\eps^{\betpar{1}}\,\min\left(\{1, \tfrac{t}{\varepsilon^{1+\alpha}}\}\right)\,\|\Apar h\|^2\\[1em]
	&+&\ds
	\gampar{2}\,\eps^{\betpar{2}}\,\min\left(\{1, \tfrac{t}{\varepsilon^{1+\alpha}}\}\right)^3\,\|\Cpar h\|^2
	\,+\,2\,\gampar{3}\,\eps^{\betpar{3}}\,\min\left(\{1, \tfrac{t}{\varepsilon^{1+\alpha}}\}\right)^2\,\langle\Apar h,\Cpar h\rangle\\[1em]
	&+&\ds
	\gamperp{1}\,\eps^{\betperp{1}}\,\min\left(\{1, \tfrac{t}{\varepsilon^{1+\alpha}}\}\right)\,\|\Aperp h\|^2
	\,+\,\gamperp{2}\,\eps^{\betperp{2}}\,\min\left(\{1, \tfrac{t}{\varepsilon^{1+\alpha}}\}\right)^3\,\|\Cperp h\|^2\\[1em]
	&+&2\,\gamperp{3}\,\eps^{\betperp{3}}\,\min\left(\{1, \tfrac{t}{\varepsilon^{1+\alpha}}\}\right)^2\,\langle\Aperp h,\Cperp h\rangle
      \end{array}
      $$
      and the corresponding $\eps$ and time dependent (partial) dissipation
      $$
      \begin{array}{rcl}
	\dnorm{\eps}{t}{h}&=&\ds
	\eps^{-(1+\alpha)}\|\nabla_vh\|^2
	\,+\,\eps^{-(1+\alpha)+\betpar{1}}\,\min\left(\{1, \tfrac{t}{\varepsilon^{1+\alpha}}\}\right)\,\|\nabla_v\Apar h\|^2\\[1em]
	&+&\ds
	\,\eps^{-(1+\alpha)+\betpar{2}}\,\min\left(\{1, \tfrac{t}{\varepsilon^{1+\alpha}}\}\right)^3\,\|\nabla_v\Cpar h\|^2
	\,+\,\eps^{-1+\betpar{3}}\,\min\left(\{1, \tfrac{t}{\varepsilon^{1+\alpha}}\}\right)^2\,\|\Cpar h\|^2\\[1em]
	&+&\ds
	\,\eps^{-(1+\alpha)+\betperp{1}}\,\min\left(\{1, \tfrac{t}{\varepsilon^{1+\alpha}}\}\right)\,\|\nabla_v\Aperp h\|^2
	\,+\,\eps^{-(1+\alpha)+\betperp{2}}\,\min\left(\{1, \tfrac{t}{\varepsilon^{1+\alpha}}\}\right)^3\,\|\nabla_v\Cperp h\|^2\\[1em]
	&+&\,\eps^{-1+\betperp{3}}\,\min\left(\{1, \tfrac{t}{\varepsilon^{1+\alpha}}\}\right)^2\,\|\Cperp h\|^2\,.
      \end{array}
      $$
      We prove an extended version of Lemma~\ref{l:technical} under 
      \begin{equation}
	\begin{array}{rcccl}\ds
	  \max(\{2-\alpha,\ (\betperp{1} + \betperp{2})/2\})
	  &\leq&\betperp{3}&\leq&\ds
	  \min(\{\alpha + 2 \betperp{1},\ \alpha + 2 \betperp{2},\ \betperp{2} -\alpha\}),\\[1em]\ds
	  \max(\{|\alpha|,\ (\betpar{1} + \betpar{2})/2)
	  &\leq&\betpar{3}&\leq&\ds
	  \min(\{\alpha + 2 \betpar{1},\ \alpha + 2 \betpar{2},\ \betpar{2} -\alpha\}).
	\end{array}
	\label{condbeta}
      \end{equation}
      The actual choice follows from minimization of $\betpar{3}$ and $\betperp{3}$ under this constraint that aims at the best possible control on $\d_{x_\mypar}h^\eps$ and $\nabla_{x_\perp}h^\eps$ through the dissipation. Indeed, \eqref{condbeta} requires $\betpar{3}\geq|\alpha|$. And the optimal choice $\betpar{3}=|\alpha|$ enforces
      \[
      \betpar{1} + \betpar{2} \leq 2|\alpha|,\qquad \betpar{2}-\alpha\geq|\alpha|, \qquad 2\betpar{1}+\alpha\geq|\alpha|, \qquad 2\betpar{2}+\alpha\geq|\alpha|
      \]
      which uniquely determine $\betpar{1}$ and $\betpar{2}$ as $\betpar{1} = 0$, $\betpar{2} = 2\alpha$ when $\alpha\geq 0$ and $\betpar{1} = -\alpha$, $\betpar{2} = -\alpha$ otherwise. The optimal choice for $\betperp{3}$ is $\betperp{3} = 2-\alpha$ that is available provided that $\betperp{1} \geq 1-\alpha$, $\betperp{2} \geq 2$ and $\betperp{1}+\betperp{2}\leq 4-2\alpha$. The latter inequalities do not determine $\betperp{1}$ and $\betperp{2}$ uniquely but the best choice is indeed $\betperp{1}=1-\alpha$, $\betperp{2}=2$.
      \bpr
      Let $\phi\in W^{2,\infty}(\T^3)$. Under condition~\eqref{condbeta} there exist positive constants $\gamperp{i}$, $\gampar{i}$, $i\in\{1,2,3\}$ such that, 
      \begin{itemize}
	\item uniformly in $(t,\eps)\in\R_+\times(0,1)$, the norm $\norm{\eps}{t}{\cdot}$ is equivalent to the square root of
	$$
	\begin{array}{rcl}
	  \| \cdot\|^2
	  &+&\ds
	  \eps^{\betpar{1}}\,\min\left(\{1, \tfrac{t}{\varepsilon^{1+\alpha}}\}\right)\,\|\Apar (\cdot)\|^2
	  \,+\,\eps^{\betpar{2}}\,\min\left(\{1, \tfrac{t}{\varepsilon^{1+\alpha}}\}\right)^3\,\|\Cpar(\cdot)\|^2\\[1em]
	  &+&\ds
	  \eps^{\betperp{1}}\,\min\left(\{1, \tfrac{t}{\varepsilon^{1+\alpha}}\}\right)\,\|\Aperp(\cdot)\|^2
	  \,+\,\eps^{\betperp{2}}\,\min\left(\{1, \tfrac{t}{\varepsilon^{1+\alpha}}\}\right)^3\,\|\Cperp(\cdot)\|^2\,.
	\end{array}
	$$
	\item there exists $K>0$ such that for any $h_0\in\cH$ the solution $h^\eps$ of~\eqref{FPh} starting from $h_0$ satisfies \eqref{estiml2dissipani} for any $\eps\in(0,1)$ and any $t\geq0$.
      \end{itemize}
      \label{hypocoercpar}
      \epr
      To emphasize the algebraic nature of the proof, as in \cite{Villani_hypo} we introduce abstract notation for operators
      \[
      A\ =\ \nabla_v\,,\qquad
      B\ =\ v\cdot\nabla_x -\sigma\,\nabla_x\phi\cdot\nabla_v,
      \]
      and their coordinate-wise adjoints in $\cH$,
      \[
      A^*\ =\ (v-\nabla_v) \,,\qquad B^* = -B.
      \]
      Additionally, we evaluate (and define) the following commutators
      \[
      C\ :=\ [A,B] = \nabla_x\,,\qquad
      [B,C] = \sigma\text{Hess}(\phi)\,\nabla_v\,,\qquad
      [A_i,A^*_j] = \delta_{ij}\,,\qquad
      [A, v^\perp\cdot\nabla_v]\ =\ -A^\perp 
      \]
      where $\delta_{ij}$ is the classical Kronecker symbol. Also, as a preliminary computation we observe, using Einstein's summation convention on repeated indices, that
      \[
      \begin{aligned}
	&\|A^*\cdot Ah\|^2&=&&&\lla A^*_iA_ih, A^*_jA_jh \rra\\
	&&=&&&\lla A_jA^*_iA_ih, A_jh \rra\\
	&&=&&&\lla \delta_{ij}A_ih, A_jh \rra + \lla A_jA_ih, A_iA_jh \rra\\
	&&=&&&\|Ah\|^2 + \|A^2h\|^2,
      \end{aligned}
      \]
      where $A^2$ denotes the matrix $A\otimes A$. 
      Finally we set 
      \[
      L_\eps = \frac{1}{\varepsilon^{\alpha+1}}\,A^*\cdot A + \frac{1}{\varepsilon}\,B.
      \]
      to write \eqref{FPh} as
      \[
      \left(\partial_t + L_\eps - \frac{\sigma}{\varepsilon^2}\,v^\perp\cdot\nabla_v\right)h^\eps\,=\,0.
      \]
      \bl
      Let $h^\eps$ solves \eqref{FPh}. Then equations \eqref{estvxpar} and \eqref{estvxperp} hold as well as
      \begin{align}
	&\frac 12\frac{\dD}{\dD t}\|h^\eps\|^2  + \frac{1}{\varepsilon^{1+\alpha}}\|Ah^\eps\|^2\,=\,0,\label{est0}\\
	&\frac 12\frac{\dD}{\dD t}\|A_\# h^\eps\|^2 + \frac{1}{\varepsilon^{1+\alpha}}(\|A_\# h^\eps\|^2 + \|A_\#Ah^\eps\|^2)\ =\ -\frac{1}{\varepsilon}\lla C_\#h^\eps, A_\#h^\eps\rra,\label{estv}\\
	&\frac 12\frac{\dD}{\dD t}\|C_\# h^\eps\|^2 + \frac{1}{\varepsilon^{1+\alpha}}\|C_\# Ah^\eps\|^2\ =\ \frac{\sigma}{\varepsilon}\lla (Ah^\eps\cdot\nabla_x)\nabla_{x_\#}\phi,C_\# h^\eps\rra,\label{estx}
      \end{align}
      where $\#$ stands either for $\perp$ or $\mypar$.
      \label{lemhypo}
      \el
      
      \begin{proof}
	As already pointed out, estimate \eqref{est0} follows from \eqref{e:H} with $H(h)=h^2$. We prove \eqref{estv} by computing on one hand for any $i$
	\[
	\lla A_iL_\eps h^\eps,A_ih^\eps\rra
	\ =\ \frac{1}{\varepsilon^{1+\alpha}}(\|A_ih^\eps\|^2 + \|A_i Ah^\eps\|^2)+\frac{1}{\varepsilon}\lla C_ih^\eps,A_ih^\eps\rra
	\]
	and on the other hand
	\[
	\lla A_i v^\perp\cdot\nabla_v h^\eps,A_ih^\eps\rra\ =\ -\lla (A^\perp)_i h^\eps, A_ih^\eps\rra
	\]
	which vanishes when either one sums over $i=1,2$ or takes $i=3$. Equality \eqref{estx} follows from the fact that for any $i$
	\[
	\lla C_iL_\eps h^\eps,C_ih^\eps\rra\,=\ \frac{1}{\varepsilon^{1+\alpha}}\|C_i Ah^\eps\|^2 - \frac{\sigma}{\varepsilon}\lla (Ah^\eps\cdot\nabla_x)\d_{x_i}\phi,C_ih^\eps\rra
	\]
	and the fact that the $v^\perp\cdot\nabla_v$ commutes with space derivatives hence its contribution vanishes by skew-symmetry. To derive \eqref{estvxpar} and \eqref{estvxperp} we first observe that for any $i$
	$$
	\begin{array}{rcl}\ds
	  \Big\langle C_iL_\eps h^\eps\hspace{-1em}&\hspace{-1em},\hspace{-1em}&\hspace{-1em}\ds
	  A_ih^\eps\Big\rangle+\lla A_iL_\eps h^\eps,C_ih^\eps\rra\\[0.5em]
	  &=&\ds
	  \frac{1}{\eps^{1+\alpha}}\lla C_i h^\eps,A_ih^\eps\rra 
	  +\frac{2}{\eps^{1+\alpha}}\lla AC_i h^\eps,A A_ih^\eps\rra 
	  -\frac{\sigma}{\eps}\lla (Ah^\eps\cdot\nabla_x)\d_{x_i}\phi,A_ih^\eps\rra
	  +\|C_ih^\eps\|^2
	\end{array}
	$$
	and conclude with for any $i$
	$$
	\lla C_i v^\perp\cdot\nabla_v h^\eps,A_ih^\eps\rra +\lla A_i v^\perp\cdot\nabla_v h^\eps,C_ih^\eps\rra
	\ =\ -\lla (A^\perp)_i h^\eps, C_ih^\eps\rra\,.
	$$
      \end{proof}

      \begin{proof}[Proof of Proposition \ref{hypocoercpar}]
	By using estimates  \eqref{est0}, \eqref{estx}, \eqref{estv}, \eqref{estvxperp} and \eqref{estvxpar} in each direction we infer that\footnote{We keep the variable $t$ essentially implicit in the estimate for concision's sake.}, for any $\eps\in(0,1)$, $\beta\in\R^6$ and $\gamma\in(\R_+^*)^6$
	$$
	\begin{array}{rcl}
	  &\ds\frac12\hspace{-0.5em}&\ds
	  \hspace{-0.5em}\frac{\dD}{\dD t}\left(\|h^\eps\|_{\eps,\cdot}^2\right)\\[0.5em]
	  &+&\ds
	  \sum_{\#\in\{\perp,\mypar\}}\Big[
	  \eps^{-1-\alpha}\|A_\#h^\eps\|^2
	  +\gamma_{1\#}\min\left(1, \tfrac{t}{\eps^{1+\alpha}}\right)\eps^{\beta_{1\#}-1-\alpha}(\|A_\#h^\eps\|^2 + \|A_\#Ah^\eps\|^2)\\[0.5em]
	  &&\ds
	  \phantom{\sum_{\#\in\{\perp,\mypar\}}}+
	  \gamma_{2\#}\min\left(1, \tfrac{t}{\eps^{1+\alpha}}\right)^3\eps^{\beta_{2\#}-1-\alpha}\|C_\#Ah^\eps\|^2
	  + \gamma_{3\#}\min\left(1, \tfrac{t}{\eps^{1+\alpha}}\right)^2\eps^{\beta_{3\#}-1}\|C_\#h^\eps\|^2\Big]
	  \\[0.5em]
	  &\leq&\ds
	  \sum_{\#\in\{\perp,\mypar\}}\Big[
	  \gamma_{1\#}\min\left(1, \tfrac{t}{\eps^{1+\alpha}}\right)\eps^{\beta_{1\#}-1}\|C_\#h^\eps\|\|A_\#h^\eps\|\\[0.5em]
	  &&\ds
	  \phantom{\sum_{\#\in\{\perp,\mypar\}}}+
	  \gamma_{2\#}\min\left(1, \tfrac{t}{\eps^{1+\alpha}}\right)^3\|\text{Hess}\phi\|_{L^\infty(\T^3)} \varepsilon^{\beta_{2\#}-1}\|C_\#h^\eps\|\|Ah^\eps\|\\[0.5em]
	  &&\ds
	  \phantom{\sum_{\#\in\{\perp,\mypar\}}}+
	  \gamma_{3\#}\min\left(1, \tfrac{t}{\eps^{1+\alpha}}\right)^2\|\text{Hess}\phi\|_{L^\infty(\T^3)} \varepsilon^{\beta_{3\#}-1}\|A_\#h^\eps\|\|Ah^\eps\|\\[0.5em]
	  &&\ds
	  \phantom{\sum_{\#\in\{\perp,\mypar\}}}+
	  \gamma_{3\#}\min\left(1, \tfrac{t}{\eps^{1+\alpha}}\right)^2\eps^{\beta_{3\#}-1-\alpha}\left(\|A_\#h^\eps\|\| C_\#h^\eps\| + 2\|A_\#Ah^\eps\|\|C_\#Ah^\eps\|\right)\Big]\\[0.5em]
	  &+&\ds
	  \gamma_{3\perp}\min\left(1, \tfrac{t}{\varepsilon^{1+\alpha}}\right)^2\varepsilon^{\betperp{3}-2}\|A_\perp h^\eps\|\|C_\perp h^\eps\|\\[0.5em]
	  &+&\ds
	  \tfrac12\sum_{\#\in\{\perp,\mypar\}}\Big[
	  \gamma_{1\#}\varepsilon^{\beta_{1\#}-1-\alpha}\|A_\#h^\eps\|^2
	  +3\gamma_{2\#}\varepsilon^{\beta_{2\#}-1-\alpha}\min\left(1, \tfrac{t}{\varepsilon^{1+\alpha}}\right)^2\|C_\#h^\eps\|^2\\[0.5em]
	  &&\ds
	  \phantom{\sum_{\#\in\{\perp,\mypar\}}}\quad+
	  4\gamma_{3\#}\varepsilon^{\beta_{3\#}-1-\alpha}\min\left(1, \tfrac{t}{\varepsilon^{1+\alpha}}\right)\|A_\#h^\eps\|\|C_\# h^\eps\|\Big{]}
	\end{array}
	$$
	where the last sum arises from differentiation of time factors in the definition of $\|h^\eps(t,\cdot,\cdot)\|_{\eps,t}$. In the former inequality, the right-hand side may be controlled by the dissipation on the left-hand side using the following procedure. Young's inequalities implies that for any $(a,b,c)\in\R^3$ and $(K_a,K_b,K_c)\in(\R_+^*)^3$ there exists $K_0>0$ such that for any positive $N_1, N_2$ and any $\eps\in(0,1)$ 
	\[
	-K_a\ \eps^a N_1^2 -K_b\ \eps^b N_2^2 + 2K_c\  \eps^c N_1N_2 \leq -K_0( \eps^a N_1^2 + \eps^b N_2^2),
	\]
	provided that\footnote{One may relax conditions to 
	  \[
	  2c>a+b\qquad\text{or}\qquad\left(2c=a+b\ \text{and}\ K_c<K_aK_b\right)
	  \]
	  if one replaces $\eps\in(0,1)$ with $\eps\in(0,\eps_0)$ for some sufficiently small $\eps_0$ depending on $(a,b,c)$ and $(K_a,K_b,K_c)$.}
	\[
	2c\geq a+b\qquad\text{and}\qquad K_c<K_aK_b\,.
	\]
	It imposes the following conditions 
	on $\beta$
	\[
	\begin{aligned}
	  \left.\begin{array}{rcl}
	    \ds\min(\{\,\alpha + 2 \beta_{1\#}\,,\,\alpha + 2 \beta_{2\#}\,\})&\geq&\ds\beta_{3\#}\\[0.5em]
	    \ds\max(\{\,\alpha\,,\,-\alpha,\,(\beta_{1\#} + \beta_{2\#})/2\,\})&\leq&\beta_{3\#}
	  \end{array}\right\}&\text{ Induced by Fokker-Planck}\\
	  \betperp{3}\ \geq\  2-\alpha\qquad&\text{ Induced by magnetic field}\\
	  \left.\begin{array}{rcl}
	    \beta_{1\#}&\geq&0\\[0.5em]
	    \beta_{2\#}-\alpha&\geq&\beta_{3\#}\\[0.5em]
	    \alpha&\leq&\beta_{3\#}
	  \end{array}\right\}&\text{ Due to regularization in time}
	\end{aligned}
	\]
	By eliminating redundant conditions we obtain \eqref{condbeta}. It remains to find $\gamma$s satisfying the remaining constraints ''$K_c<K_aK_b$''. To do so we seek $\gamma$s as $\gamma_{i\#} = \eta^{c_{i}}$ with $\eta$ positive and sufficiently small. Remaining conditions reduce then to
	\[
	2c_1>c_3 + 0, \qquad 2c_2>c_3 + 0, \qquad c_3 > 0, \qquad 2c_3>c_3 + 0, \qquad 2c_3>c_1 + c_2, 
	\]
	\[
	c_3>0,\qquad c_1>0, \qquad c_2>c_3, \qquad c_3>0
	\]
	and a smallness condition on $\eta$ that we do not explicit here. To achieve the proof note that the former condition on $c\,$s is satisfied with for instance $c_1 = 1$, $c_2=2$ and $c_3 = 7/4$. At last we observe that those conditions are also sufficient to ensure by similar arguments the claimed equivalence of norms. 
      \end{proof}

      \section{Maximum principle on the torus}\label{s:maxprinc}
      
      For convenience we state and prove here maximum principles adapted to our purposes.
      
      \bl[Maximum principle on the torus, three-dimensional case] There exists a constant $C$ such that if $u\in W^{1,1}(\TT^3)$ and $-\Delta u\leq K_0$ with $K_0\in L^{3/2,1}(\TT^3)$ then 
      $$
      \esssup_{\TT^3} u\,\leq\,C\,(\|u\|_{W^{1,1}(\TT^3)}+\|K_0\|_{L^{3/2,1}(\TT^3)})\,.
      $$
      \label{maxprinc}
      \el
      
      We refer for instance to \cite[Chapter~2]{Lemarie-Rieusset} for relevant basic properties of Lorentz spaces appearing in the statement of the former lemma.
      
      The foregoing lemma is readily derived from the following one through the canonical identification of functions on $\TT^3$ and periodic functions on $\R^3$ by choosing $\Omega$ containing a fundamental domain and $\Omega'$ a suitably larger bounded domain.
      
      \bl[Maximum principle, three-dimensional case] 
      Let $\Omega'$ be an open subset of $\R^3$ and $\Omega$ be a bounded open subset of $\Omega$ such that $\overline{\Omega}\subset\Omega'$.\\ 
      There exists a constant $C$ such that if $u\in W^{1,1}(\Omega')$ and $-\Delta u\leq K_0$ with $K_0\in L^{3/2,1}(\Omega')$ then 
      $$
      \esssup_{\Omega} u\,\leq\,C\,(\|u\|_{W^{1,1}(\Omega')}+\|K_0\|_{L^{3/2,1}(\Omega')})\,.
      $$
      \label{maxprinc-bis}
      \el
      
      \begin{proof}
	It is sufficient to prove the estimate when $u$ and $K_0$ are smooth. From there a classical approximation argument yields the full result by providing upper bounds at any Lebesgue point of $u$.
	
	For any $0\leq \eps <\delta$ and $x_0\in\R^3$, we set
	$$
	\Omega_{\eps,\delta}(x_0)\,=\,\{\quad y\in\R^3\quad|\quad\eps<\|y-x_0\|<\delta\quad\}
	$$ 
	and for any $\eps>0$ and $x_0\in\R^3$, we denote 
	$$
	S_{\eps}(x_0)\,=\,\{\quad y\in\R^3\quad|\quad\|y-x_0\|=\eps\quad\}\,.
	$$
	We also consider $G:\R^3\to\R$, $y\mapsto \|y\|^{-1}$, which is a multiple of the Green function of the Laplacian on $\R^3$. 
	
	First we choose $R>0$ such that for any $x_0\in\Omega$, $\overline{\Omega_{0,\delta}(x_0)}\subset \Omega'$. Then, for any $x_0\in\Omega$, repeated use of the Green formula provides for any $0<\eps<R$
	$$
	\begin{array}{rcl}
	  \ds
	  \frac{1}{\eps^2}\int_{S_\eps(x_0)}u
	  &=&\ds
	  \frac{1}{R^2}\int_{S_R(x_0)}u\,+\,\left(\frac{1}{\eps}-\frac{1}{R}\right)\int_{\Omega_{0,\eps}(x_0)}(-\Delta u)
	  \,+\,\int_{\Omega_{\eps,R}(x_0)}\left(G(\cdot-x_0)-\frac{1}{R}\right)(-\Delta u)
	  \\
	  &\leq&\ds
	  \frac{1}{R^2}\int_{S_R(x_0)}u\,+\,\left(\frac{1}{\eps}-\frac{1}{R}\right)\int_{\Omega_{0,\eps}(x_0)}K_0
	  \,+\,\int_{\Omega_{\eps,R}(x_0)}\left(G(\cdot-x_0)-\frac{1}{R}\right)K_0
	  \\
	  &\leq&\ds
	  \frac{1}{R^2}\|u\|_{L^{1}(S_R(x_0))}\,+\,\frac{1}{\eps}\|K_0\|_{L^1(\Omega_{0,\eps}(x_0))}
	  \,+\,\left\|\left(G(\cdot-x_0)-\frac{1}{R}\right)K_0\right\|_{L^1(\Omega_{0,R}(x_0))}
	  \\
	  &\leq&\ds
	  C'\|u\|_{W^{1,1}(\Omega')}\,+\,C'\|K_0\|_{L^{3/2}(\Omega_{0,\eps}(x_0))}
	  \,+\,C'\|K_0\|_{L^{3/2,1}(\Omega')}
	\end{array}
	$$
	for some constant $C'$ depending only on $R$ (for instance through $\|G-R^{-1}\|_{L^{3,\infty}(\Omega_{0,R}(0))}$). Taking the limit $\eps\to0$ provides the estimate at $x_0$.
      \end{proof}
      
      We shall also use the foregoing slight generalization.
      
      \bl[Maximum principle, three-dimensional case, second version] 
      Let $\Omega'$ be an open subset of $\R^3$ and $\Omega$ be a bounded open subset of $\Omega$ such that $\overline{\Omega}\subset\Omega'$.\\ 
      There exists a constant $C$ such that if $u\in W^{1,1}(\Omega')$ and, for some $K_0\in L^{3/2,1}(\Omega')$, stands $\chi_{u\geq0}(-\Delta u)\leq K_0$ on $\Omega'$ then 
      $$
      \esssup_{\Omega} u\,\leq\,C\,(\|u\|_{W^{1,1}(\Omega')}+\|K_0\|_{L^{3/2,1}(\Omega')})\,.
      $$
      \label{maxprinc-ter}
      \el
      
      \begin{proof}
	Since $u\leq u_+$ (where $(\cdot)_+$ denotes positive part) the proof is achieved by applying Lemma~\ref{maxprinc-bis} to $u_+$ since $\nabla u_+=\chi_{u>0}\nabla u$ and 
	$$-\Delta u_+ \leq \chi_{u>0}(-\Delta u)$$
	as is classical and derived by inspecting the limit $\eps\to0$ of 
	$$
	(\sqrt{\eps^2+u^2}-\eps)\chi_{u\geq0}\,.
	$$
      \end{proof}
      
      At last we observe that a slight variation on the proof of Lemma~\ref{maxprinc-bis} provides the following two-dimensional versions of Lemmas~\ref{maxprinc} and~\ref{maxprinc-ter}.
      
      \bl[Maximum principle on the torus, two-dimensional case] For any $1<p\leq\infty$ there exists a constant $C$ such that if $u\in W^{1,1}(\TT^2)$ and $-\Delta u\leq K_0$ with $K_0\in L^p(\TT^2)$ then 
      $$
      \esssup_{\TT^2} u\,\leq\,C\,(\|u\|_{W^{1,1}(\TT^2)}+\|K_0\|_{L^{p}(\TT^2)})\,.
      $$
      \label{maxprinc-2D}
      \el
      
      \bl[Maximum principle, two-dimensional case, second version] 
      Let $1<p<\infty$, $\Omega'$ be an open subset of $\R^2$ and $\Omega$ be a bounded open subset of $\Omega$ such that $\overline{\Omega}\subset\Omega'$.\\ 
      There exists a constant $C$ such that if $u\in W^{1,1}(\Omega')$ and $\chi_{u\geq0}(-\Delta u)\leq K_0$ on  with $K_0\in L^p(\Omega')$ then 
      $$
      \esssup_{\Omega} u\,\leq\,C\,(\|u\|_{W^{1,1}(\Omega')}+\|K_0\|_{L^p(\Omega')})\,.
      $$
      \label{maxprinc-2D-bis}
      \el
      \bibliography{gyroVPFP.bib}
      \bibliographystyle{plain}
    \end{document}